\documentclass[11pt]{amsart}
\usepackage{amssymb,amsmath,amsthm}
\usepackage[mathscr]{euscript}
\usepackage{enumerate, verbatim, url}
\newtheorem{theorem}{Theorem}
\newtheorem{lemma}[theorem]{Lemma}

\newtheorem{proposition}[theorem]{Proposition}
\newtheorem{definition}[theorem]{Definition}

\theoremstyle{remark}

\newtheorem*{remark}{Remark}
\newtheorem*{example}{Example}
\numberwithin{theorem}{section} \numberwithin{equation}{section}

\newcommand{\calO}{\mathcal{O}}

\newcommand{\C}{\mathbb{C}}
\newcommand{\F}{\mathbb{F}}
\newcommand{\Cl}{{\text {\rm Cl}}}
\newcommand{\Tr}{{\text {\rm Tr}}}

\newcommand{\Q}{\mathbb{Q}}

\newcommand{\Z}{\mathbb{Z}}

\newcommand{\SL}{{\text {\rm SL}}}
\newcommand{\GL}{{\text {\rm GL}}}
\newcommand{\textmod}{{\text {\rm mod}}}

\newcommand{\calP}{\mathcal{P}}

\newcommand{\lcm}{{\text {\rm lcm}}}

\newcommand{\add}{{\text {\rm add}}}
\newcommand{\sub}{{\text {\rm sub}}}

\newcommand{\Disc}{{\text {\rm Disc}}}
\newcommand{\Stab}{{\text {\rm Stab}}}

\newcommand{\sump}{\sideset{}{'}\sum}

\newcommand{\Res}{\textnormal{Res}}
\newcommand{\calS}{\mathcal{S}}
\newcommand{\Aut}{\textnormal{Aut}}

\newcommand{\ctt}{\textnormal{ct}}

\begin{document}
\title[Counting Cubic Fields]
{Secondary Terms in Counting Functions for Cubic Fields}
\author{Takashi Taniguchi}
\address{Department of Mathematics, Graduate School of Science,
Kobe University, 1-1, Rokkodai, Nada-ku, Kobe 657-8501, Japan}
\address{Department of Mathematics, Princeton University, Fine Hall,
Washington Road, Princeton, NJ 08540}
\email{tani@math.kobe-u.ac.jp}

\author{Frank Thorne}
\address{Department of Mathematics, Stanford University, Stanford, CA 94305}
\address{Department of Mathematics, University of South Carolina,
1523 Greene Street, Columbia, SC 29208}
\email{fthorne@math.stanford.edu}

\begin{abstract}
We prove the existence of secondary terms of order $X^{5/6}$ in the Davenport-Heilbronn theorems on cubic fields
and $3$-torsion in class groups of quadratic fields. For cubic fields this confirms a conjecture of Datskovsky-Wright
and Roberts. 
We also prove a variety of generalizations, including to arithmetic progressions, where we
discover a curious bias in the secondary term.

Roberts' conjecture has also been proved independently by Bhargava, Shankar, and Tsimerman. In contrast
to their work, our proof uses the analytic theory of zeta functions associated to the space of
binary cubic forms, developed by Shintani and Datskovsky-Wright.
\end{abstract}

%
%


\maketitle
\section{Introduction}
Let $N_3^{\pm}(X)$ count the number of cubic fields $K$ with $0 < \pm \Disc(K) < X$.
In \cite{R}, Roberts conjectured that
\begin{equation}\label{conj_rc}
N_3^{\pm} (X) = C^{\pm} \frac{1}{12 \zeta(3)} X + K^{\pm} 
\frac{4 \zeta(1/3)}{5 \Gamma(2/3)^3 \zeta(5/3)} X^{5/6} + o(X^{5/6}),
\end{equation}
where $C^- = 3$, $C^+ = 1$, $K^- = \sqrt{3}$, and
$K^+ = 1$. This conjecture also appeared implicitly in an earlier paper of Datskovsky and Wright \cite{DW3}.
It was based on a combination of numerical and theoretical evidence, the latter
of which will be described in due course.

In this paper we will prove the conjecture, with an additional power savings in the error term:
\begin{theorem}\label{thm_rc}
Roberts' conjecture is true. Indeed, we have
\begin{equation}\label{eqn_main_thm}
N_3^{\pm} (X) = C^{\pm} \frac{1}{12 \zeta(3)} X + K^{\pm} 
\frac{4 \zeta(1/3)}{5 \Gamma(2/3)^3 \zeta(5/3)} X^{5/6} + O(X^{7/9 + \epsilon}).
\end{equation}
\end{theorem}

The main term is due to Davenport and Heilbronn \cite{DH}, and \eqref{eqn_main_thm} improves upon a result of Belabas, Bhargava, and Pomerance \cite{BBP}, who 
obtained an error term of $O(X^{7/8 + \epsilon})$. Although their methods were not 
designed to extract
the secondary term in \eqref{eqn_main_thm}, our approach nevertheless 
owes a great deal to theirs.

\begin{remark} Roberts' conjecture has also been proved, with an error term of $O(X^{13/16 + \epsilon})$, 
in recent and independent work of Bhargava, Shankar, and Tsimerman \cite{BST}. Their proof also allows
local specifications, as in our Theorem \ref{thm_rc_extended}.

\end{remark}
Our methods extend to the related problem of counting 3-torsion in quadratic fields. For any quadratic field
with discriminant $D$, let $\Cl_3(D)$ denote the 3-torsion subgroup of the ideal class group
$\Cl(\Q(\sqrt{D}))$. We will prove the following result.
\begin{theorem}\label{thm_rc_torsion}
We have
\begin{equation}\label{eqn_torsion1}
\sum_{0 < \pm D < X} \# \Cl_3(D) = \frac{3 + C^{\pm}}{\pi^2} X + K^{\pm}
\frac{8 \zeta(1/3)}{5 \Gamma(2/3)^3} \prod_p \Bigg(1 - \frac{p^{1/3} + 1}{p (p + 1)} \Bigg)
 X^{5/6} + O(X^{18/23 + \epsilon}),
\end{equation}
and
where the sum ranges over fundamental discriminants $D$, the product is over all primes, and the constants are as before.
\end{theorem}
As with \eqref{eqn_main_thm}, the main term is due to Davenport and Heilbronn \cite{DH}, and an error term of $O(X^{7/8 + \epsilon})$ 
was proved by Belabas, Bhargava, and Pomerance \cite{BBP}.
Our error term is slightly higher than that in \eqref{eqn_main_thm} due to an additional technical complication which appears in the proof.

\begin{remark} 
Extensive 
computational results for the 3-parts of class groups appear in \cite{JRW, bel}. For smaller $X$ it is now practical to check 
\eqref{eqn_torsion1}
numerically using PARI/GP \cite{pari}. For example,
for $D > 0$ and $X = 10^6$, the left side of \eqref{eqn_torsion1} is $381071$, and the main terms on the right sum to
$381337.24\cdots$. For $D < 0$ these values are $566398$ and $566448.83\cdots$ respectively.

\end{remark}

As conjectured by Roberts, our results also extend to counting problems where various local restrictions are imposed.
This is perhaps most interesting in the case of counting fields. Let $\mathcal{S} = (\mathcal{S}_p)$
be a finite set of {\itshape local specifications}. In particular, we may require that $K$ be
inert, partially ramified, totally ramified, partially split,
or totally split at $p$. More generally, we may specify the $p$-adic completion $K_p$.

We will prove the following quantitative version of Roberts' extended conjecture:
\begin{theorem}\label{thm_rc_extended}
With the notation above, the number of cubic fields $K$ satisfying $\mathcal{S}$ with $0 < \pm \Disc(K) < X$ is 
\begin{equation}\label{eqn_main_thm_extended}
N_3^{\pm} (X, \mathcal{S}) = C^{\pm} (\mathcal{S}) \frac{1}{12 \zeta(3)} X + K^{\pm}(\mathcal{S})
\frac{4 \zeta(1/3)}{5 \Gamma(2/3)^3 \zeta(5/3)} X^{5/6} + O\Big( X^{7/9 + \epsilon} \prod_p p^{8 e_p/9} \Big),
\end{equation}
where $e_p = 0$ if there is no specification at $p$, $e_p = 1$ if we only count fields unramified at $p$, and $e_p = 2$ otherwise,
and the constants $C^{\pm}(\mathcal{S})$ and $K^{\pm}(\mathcal{S})$ are computed explicitly in Section \ref{sec_generalizations}.
\end{theorem}

We obtain a similar generalization of Theorem \ref{thm_rc_torsion}. In this case, we may restrict our sum to $D$ for which
finitely many primes $p$ are inert, split, or ramified. In the ramified case we may also specify the completion of
$\Q(\sqrt{D})$ at $p$. We will prove the following result:
\begin{theorem}\label{thm_rc_torsion2}
With the notation above, we have
\begin{equation}\label{eqn_torsion3}
\sump_{0 < \pm D < X} \# \Cl_3(D) = \frac{3 + C^{\pm}}{\pi^2} C'(\mathcal{S}) X +  K'^{\pm}(\mathcal{S})
\frac{8 \zeta(1/3)}{5 \Gamma(2/3)^3} \prod_p \Bigg(1 - \frac{p^{1/3} + 1}{p (p + 1)} \Bigg) X^{5/6} + O\Big(X^{18/23 + \epsilon} \prod_p 
p^{20 e_p/23} \Big),
\end{equation}
where the sums are restricted to discriminants meeting the conditions specified by $\mathcal{S}$, and the 
constants $C'(\mathcal{S})$ and $K'^{\pm}(\mathcal{S})$ are computed explicitly in Section \ref{sec_generalizations}.
\end{theorem}

Finally, our results allow us to count discriminants in arithmetic progressions. Here we encounter a curious phenomenon.
Consider the following table of discriminants of cubic fields $K$ with
$0 < \Disc(K) < 2 \cdot 10^6$, arranged by their residue class modulo 7:
\begin{center}
\begin{tabular}{c | c }
Discriminant modulo 7 & \textnormal{Count} \\ \hline
0 & 15330 \\ \hline
1 & 17229 \\ 
2 & 14327 \\ 
3 & 15323 \\ 
4 & 17027 \\ 
5 & 18058 \\ 
6 & 15150 \\ 
\end{tabular}
\end{center}
This data shows a striking lack of equidistribution, and as related experiments confirm, this is not a fluke. The primary
term in the counting function is the same for each residue class other than 0, but 
we will
prove that the secondary term in the counting function is different for each residue class modulo 7.
More generally, when there exist cubic
Dirichlet characters modulo $m$ (equivalently, when $m$ is divisible by 9 or any prime $\equiv 1 \ (\textmod \ 6)$), the secondary term depends on these characters in a subtle way, and we obtain biases in the distribution of fields in progressions modulo $m$. 

Our general result (Theorem \ref{thm_technical_ap}) also allows local specifications and is rather complicated to state; the following is a special case.

\begin{theorem}\label{thm_progressions}
Suppose that $(6a, m) = 1$. Then the number of cubic fields $K$ with $0 < \pm \Disc(K) < X$ and
$\Disc(K) \equiv a \ (\textmod \ m)$ is
\begin{equation}\label{eqn_twisted}
N_3^{\pm}(X; m, a) = 
\frac{C^{\pm}}{12 \zeta(3) m} \prod_{p | m} \frac{1}{1 - p^{-3}} X
+ \\ 
K_1(m, a) \frac{4K^\pm}{5\Gamma(2/3)^3} 
 X^{5/6} + O(X^{7/9 + \epsilon} m^{8/9}),
\end{equation}
where
\begin{equation}\label{eqn_thm_k1}
K_1(m, a) := 
\frac{1}{m}
\prod_{p | m} \frac{1}{1 - p^{-2}}
\sump_{\chi^6 = 1} 
\overline{\chi}(a) \frac{L(1/3, \chi^{-2})}{L(5/3, \chi^2)}
\prod_{\substack{{p | m} \\ p \nmid \textnormal{cond}(\chi)}}
 \frac{1 - \chi(p)^{-2} p^{-4/3}}{1 - \chi(p)^2 p^{-5/3}}
\prod_{\substack{{p | m} \\ p | \textnormal{cond}(\chi)}}
 \frac{\tau_p(\chi_p^2)^3}{p^2}.
 \end{equation}
Here $\tau_p(\chi_p^2) = \sum_{t \in (\Z / p \Z)^{\times}} \chi_p^2(t) e^{2 \pi i t/p}$, and
the sum is over primitive characters $\chi$ to moduli dividing $m$ (including the trivial character modulo $1$), 
such that if we write $\chi = \prod_{p | \textnormal{cond}(\chi)} \chi_p$ with each $\chi_p$ of conductor $p$, then each $\chi_p$ has exact order $6$.
\end{theorem}
We illustrate our result with numerical data for $m = 5$ and $m = 7$. We consider the number 
of fields $K$ with $0 < \Disc(K) < 2 \cdot 10^6$ and $\Disc(K) \equiv a \ (\textmod \ m)$, and for each modulus we list 
the sum of the two main terms of \eqref{eqn_twisted} (after rounding) as well as the actual numerical data. For $a = 0$, the two main terms come from
\eqref{eqn_main_thm_extended}
instead of \eqref{eqn_twisted}.

\begin{center}
\begin{tabular}{l | c | c | c | c | c  }
Discriminant modulo 5 & 0 & 1 & 2 & 3 & 4  \\ \hline
Result from \eqref{eqn_twisted} & 21307 & 22757 & 22757 & 22757 & 22757 \\
Actual count & 21277 & 22887 & 22751 & 22748 & 22781 \\
Difference & 30 & 130 & 6 & 9 & 24 \\
\end{tabular}
\end{center}

\begin{center}
\begin{tabular}{l | c | c | c | c | c | c | c}
Discriminant modulo 7 & 0 & 1 & 2 & 3 & 4 & 5 & 6 \\ \hline
Result from \eqref{eqn_twisted} & 15316 & 17209 & 14277 & 15316 & 17024 &  18063 & 15131 \\
Actual count & 15330 & 17229 & 14327 & 15323 & 17027 & 18058 & 15150 \\
Difference & 14 & 20 & 50 & 7 & 3 & 5 & 19 \\
\end{tabular}
\end{center}

\begin{remark}
As one might guess from the shape of \eqref{eqn_thm_k1}, we obtain results on arithmetic progressions by first obtaining estimates for
\begin{equation}
N_3^{\pm}(X, \chi) := \sum_{\substack{[K : \Q] = 3 \\ 0 < \pm \Disc(K) < X}} \chi(\Disc(K)),
\end{equation}
and then using the orthogonality relations for Dirichlet characters. Our results
have their origins in work of Datskovsky and Wright, who proved that certain related $L$-functions
(see \eqref{eqn_def_shintani_orbital}) may have a pole if $\chi^6 = 1$ but are otherwise entire.
\end{remark}

Our most general theorem is Theorem \ref{thm_technical_ap}. This yields
estimates for $N_3^{\pm}(X; m, a)$ for arbitrary values of $a$ and $m$, which are unfortunately 
complicated to state.  Note in particular that such
results are interesting (and nontrivial!) when $(a, m) > 1$.
In this case, and in particular in progressions $\equiv ap \ (\textmod \ p^2)$, we find a similar
(but not identical) bias in the secondary term.

Moreover, Theorem \ref{thm_technical_ap} also allows us to simultaneously incorporate local specifications. 
For example, we can count the number of fields which split completely at a prime $p$ and have
discriminant $\equiv a \ (\textmod \ p)$, for any quadratic residue $a$ modulo $p$.
\\
\\
Our final result concerns 3-torsion in class groups in arithmetic progressions. Our most general result
(Theorem \ref{thm_technical_torsion_ap}) is again complicated to state, so we state the following analogue of 
Theorem \ref{thm_progressions}:
\begin{theorem}\label{thm_torsion_progressions}
Suppose that $(6a, m) = 1$. Then,
\begin{equation}\label{eqn_torsion_twisted}
\sum_{\substack{0 < \pm D < X \\ D \equiv a \ (\textmod \ m)}} \# \Cl_3(D) = 
\frac{3 + C^{\pm}}{\pi^2 m} \bigg( \prod_{p | m} \frac{1}{1 - p^{-2}} \bigg) X + 
\frac{8 K^{\pm}}{5 \Gamma(2/3)^3} K'_1(m, a) X^{5/6} + O(X^{18/23} m^{20/23}),
\end{equation}
where
\begin{multline}
K'_1(m, a) = \frac{1}{m} \prod_{p | m} \frac{1}{1 - p^{-2}} 
\ \times \\ 
\sump_{\chi^6 = 1} 
\overline{\chi}(a) L(1/3, \chi^{-2})
\prod_{p \nmid m} \bigg(1 - \frac{ \chi(p)^2 p^{1/3} + 1}{p (p + 1)} \bigg)
\prod_{\substack{{p | m} \\ p \nmid \textnormal{cond}(\chi)}}
\Big(1 - \chi(p)^{-2} p^{-4/3}\Big)
\prod_{\substack{{p | m} \\ p | \textnormal{cond}(\chi)}}
 \frac{\tau_p(\chi_p^2)^3}{p^2}.
\end{multline}
As in Theorem \ref{thm_progressions}, the sum is over primitive sextic characters $\chi$ to moduli dividing $m$, 
such that $\chi_p$ is of exact order $6$ for each $p$.
\end{theorem}
As we did with Theorem \ref{thm_progressions}, we illustrate our result with numerical data modulo 5 and 7.
Here we compare the main terms of \eqref{eqn_torsion_twisted} and \eqref{eqn_torsion3} to the actual counts
of $\# \Cl_3(D)$ for $0 < D < 2 \cdot 10^6$. Note that these counts include the trivial element of $\Cl(D)$ for
each $D$.

\begin{center}
\begin{tabular}{l | c | c | c | c | c  }
Discriminant modulo 5 & 0 & 1 & 2 & 3 & 4  \\ \hline
Result from \eqref{eqn_torsion_twisted} & 126942 & 160239 &  160239  &   160239 & 160239  \\
Actual count & 126841 & 160373 & 160202 & 160252 & 160207 \\
Difference & 101 & 134 & 37 & 13 & 32 \\
\end{tabular}
\end{center}

\begin{center}
\begin{tabular}{l | c | c | c | c | c | c | c}
Discriminant modulo 7 & 0 & 1 & 2 & 3 & 4 & 5 & 6 \\ \hline
Result from \eqref{eqn_torsion_twisted} & 95095  & 113486 & 109566 & 110919 & 113345 & 114699 & 110779 \\
Actual count & 95138 & 113407 & 109506 & 110955 & 113232 & 114741 & 110898 \\
Difference & 43 & 79 & 60 & 36 & 113 & 42 & 119 \\
\end{tabular}
\end{center}

In this connection, we mention recent work of Hough \cite{H}, which provides another proof of the Davenport-Heilbronn
theorem for class groups of imaginary quadratic fields (Theorem \ref{thm_rc_torsion}). His methods naturally produce
the main and secondary terms of Theorem \ref{thm_rc_torsion}, albeit with an error term larger than $X^{5/6}$. His methods
are notable for avoiding the Delone-Faddeev correspondence (to be described below); he uses a result of Soundararajan
\cite{sound} which parameterizes nontrivial ideals of $\Cl(\Q(\sqrt{D}))$
in terms of a Diophantine equation, and he then proves
his result as a consequence of an equidistribution result for the associated Heegner points.

Although Hough's methods do not extend to counting cubic fields, they do extend to counting $k$-torsion in class groups
of imaginary quadratic fields for odd $k > 3$. He conjectures that a negative secondary term of order $X^{\frac{1}{2} + \frac{1}{k}}$
should appear. Moreover, as he is presently investigating, these methods carry over to arithmetic progressions
as well.

We also mention that the main term of \eqref{eqn_torsion_twisted} was previously obtained by Nakagawa and Horie \cite{NH}, with
an application to elliptic curves. In particular, they proved that if $D \equiv 2 \ (\textmod \ 3)$ is negative, and 
the class group of $\Q(\sqrt{D})$ has no nontrivial 3-torsion, then the elliptic curve $D y^2 = 4 x^3 -1$ has no rational points.
In particular, this gave a family of elliptic curves, a positive proportion of which have Mordell-Weil rank 0. Related ideas were
pursued in subsequent works of James \cite{james}, Vatsal \cite{vatsal}, and Wong \cite{wong}, among others. 

One naturally asks if our secondary terms should be reflected in counting functions of elliptic curves. 
Some quick numerical experiments
suggested a negative answer; for example, of James's quadratic twists 
$Dy^2 = x^3 - x^2 + 72x + 368$ with $D < 2500$ a fundamental discriminant, 371 of them have\footnote{We performed our computations
using Sage \cite{sage} with the \url{EllipticCurve.rank(proof=False)} function, so these counts are not proved correct.}
rank 0 and 389 have positive rank. Based on \eqref{eqn_torsion1}, one might guess that an excess of these curves has
rank 0, but this does not seem to happen.

\vskip 0.1in
{\bf Summary of the proofs.} The proofs of all of our results rely on the analytic theory of the {\itshape adelic Shintani zeta function}, due to Shintani \cite{S}
and Datskovsky-Wright \cite{DW1, DW2, DW3} and further developed by the present authors \cite{TT}. This contrasts with the geometric approach adopted by Bhargava and
his collaborators \cite{BBP, B, BST}. 

To keep the exposition as simple as possible, we have organized our paper around the proof of Theorem \ref{thm_rc}. Except as noted to the contrary 
(and in Section \ref{section_torsion_bounds} in particular) our discussion only pertains to Theorem \ref{thm_rc};
the discussion of the proofs of our various generalizations
is postponed to Section \ref{sec_generalizations}.

\subsection{The Davenport-Heilbronn and Delone-Faddeev correspondences}

Following the original work of Davenport and Heilbronn \cite{DH}, we begin by relating our problem to the more
tractable problem of counting certain integral
binary cubic forms. This is accomplished through the well-known correspondence of Davenport-Heilbronn
and Delone-Faddeev \cite{DF}. An elegant simplified and self-contained account of this work can be found
in a paper of Bhargava \cite{B}, so we will only briefly summarize it here.

In \cite{DH}, Davenport and Heilbronn established the main term in \eqref{eqn_main_thm} by first relating
{\itshape cubic rings} to {\itshape integral binary cubic
forms}, and then counting those cubic forms which correspond to maximal orders in cubic fields.

A {\itshape cubic ring} is a commutative ring which is free of rank 3 as
a $\Z$-module. The discriminant of a cubic ring is defined to be the determinant of the trace form
$\langle x, y \rangle = \Tr(xy)$, and the discriminant of the maximal order of a cubic field is equal to
the discriminant of the field.

The lattice of {\itshape integral binary cubic forms} is defined by
\begin{equation}\label{def_vz}
V_{\Z} := \{ a u^3 + b u^2 v + c u v^2 + d v^3 \ : a, b, c, d \in \Z \},
\end{equation}
and the {\itshape discriminant} of such a form is given by the usual equation
\begin{equation}\label{eqn_disc_formula}
\Disc(f) = b^2 c^2 - 4 a c^3 - 4 b^3 d - 27 a^2 d^2 + 18 abcd.
\end{equation}
There is a natural action of $\GL_2(\Z)$ (and also of $\SL_2(\Z)$) on $V_{\Z}$, given by
\begin{equation}
(\gamma \cdot f)(u, v) = \frac{1}{\det \gamma} f((u, v) \cdot \gamma).
\end{equation}
We call a cubic form $f$ {\itshape irreducible} if $f(u, v)$ is irreducible as a polynomial over $\Q$, and
{\itshape nondegenerate} if $\Disc(f) \neq 0$.

The Delone-Faddeev correspondence, which extends that of Davenport-Heilbronn and which was further extended by Gan, Gross, and Savin \cite{GGS} to include the degenerate case, 
is as follows:

\begin{theorem}[\cite{DF, GGS}] There is a natural, discriminant-preserving 
bijection between the set of $\GL_2(\Z)$-equivalence classes
of integral binary cubic forms and the set of isomorphism classes of cubic rings. Furthermore, under this
correspondence, irreducible cubic forms correspond to orders in cubic fields.
\end{theorem}

To count cubic fields, Davenport and Heilbronn count their maximal orders, which are exactly those orders
which are maximal at all primes $p$:
\begin{proposition}[\cite{DH, B}\label{def_up}]
Under the Delone-Faddeev correspondence, a cubic ring $R$ is maximal if any only if its corresponding
cubic form $f$ belongs to the set $U_p \subset V_{\Z}$ for all $p$, defined by the 
following equivalent conditions:
\begin{itemize}
\item
The ring $R$ is not contained
in any other cubic ring with index divisible by $p$.
\item
The cubic form $f$ is not a multiple of $p$, and there is no $\GL_2(\Z)$-transformation of
$f(u, v) = a u^3 + b u^2 v + c u v^2 + d v^3$ such that $a$ is a multiple of $p^2$ and $b$ is a multiple
of $p$.
\end{itemize}
\end{proposition}
In particular, the condition $U_p$ only depends on the coordinates of $f$ modulo $p^2$.

The proof of the main term in \eqref{conj_rc} goes as follows: One obtains an
asymptotic formula for the number of cubic rings of bounded discriminant by counting lattice points
in fundamental domains for the action of $\GL_2(\Z)$, bounded by the constraint $|\Disc(x)| < X$.
The fundamental domains may be chosen so that almost all reducible rings correspond to
forms with $a = 0$, and so these may be excluded from the count.

One then multiplies this asymptotic by the 
product of all the local densities of the sets $U_p$. This yields
a heuristic argument for the main term in \eqref{conj_rc}, and one incorporates a simple sieve to convert
this heuristic into a proof.

\begin{remark} The Davenport-Heilbronn correspondence also applies to reducible maximal cubic rings.
It is readily shown that (up to isomorphism) these are the rings $\Z \times \Z \times \Z$ or $\mathcal{O}_K \times \Z$,
where $\mathcal{O}_K$ is the ring of integers of a quadratic field. The unit element is given by $(1, 1, 1)$
or $(1, 1)$ respectively, and the discriminant is equal to 1 or $\Disc(K)$ as appropriate.
\end{remark}

\subsection{The work of Belabas, Bhargava, and Pomerance}
In \cite{BBP}, Belabas, Bhargava, and Pomerance (BBP) introduced improvements to Davenport and Heilbronn's
method, and obtained an error term of $O(X^{7/8 + \epsilon})$ in \eqref{conj_rc} (and  also in \eqref{eqn_torsion1}). They begin by observing that
\begin{equation}\label{eqn_BBP}
N_3^{\pm}(X) = \sum_{q \geq 1} \mu(q) N^{\pm}(q, X),
\end{equation}
where $N^{\pm}(q, X)$ counts the number of cubic orders of discriminant $0 < \pm D < X$ which
are nonmaximal at every prime dividing $q$.
For large $q$, BBP prove that
$N^{\pm}(q, X) \ll X 3^{\omega(q)} / q^2$
using reasonably elementary methods. Therefore, one may truncate the sum in \eqref{eqn_BBP}
to $q \leq Q$, with error $\ll X/Q^{1 - \epsilon}.$ We will use this fact in our proof as well.

For small $q$, BBP estimate $N^{\pm}(q, X)$ with explicit error terms using geometric methods.
These error terms are good enough to allow them to take the sum in \eqref{eqn_BBP} up to $(X \log X)^{1/8}$, 
which yields a final error term of $O(X^{7/8 + \epsilon})$.

In addition, their methods extend to counting {\itshape quartic} fields, where they obtain a main term of $C_4 X$ 
with error $\ll X^{23/24 + \epsilon}$.

\subsection{Shintani zeta functions and the analytic approach}
The main idea of this paper is 
to estimate a quantity related to $N^{\pm}(q, X)$ using the analytic theory of Shintani
zeta functions. The {\itshape Shintani zeta functions} associated to the space of binary cubic forms
are defined by the Dirichlet series
\begin{equation}\label{eqn_shintani_def_1}
\xi^{\pm}(s) := \sum_{x \in \SL_2(\Z) \backslash V_{\Z}} \frac{1}{|\Stab(x)|} |\Disc(x)|^{-s},
\end{equation}
where the lattice $V_{\Z}$ was
defined in \eqref{def_vz}, and the sum is over elements of positive or negative discriminant respectively. Here $\Stab(x)$, the stabilizer
of $x$ in $\SL_2(\Z)$, is always an abelian group of order 1 or 3. 

By the Delone-Faddeev correspondence, $\xi^{\pm}(s)$ is {\itshape almost} the generating series for cubic rings.
There are two differences: The Shintani zeta function counts $\SL_2(\Z)$-orbits rather than $\GL_2(\Z)$-orbits,
and it weights some of them by a factor of 1/3. As we will see, these discrepancies depend on the Galois group
of the splitting field of the cubic form, and we will be able to adjust for them later.

These series converge absolutely for $\Re(s) > 1$, and Shintani proved \cite{S} that
that these zeta functions enjoy analytic continuation and a functional equation, to be described later. It therefore
follows that we can use Perron's formula and a method of Landau \cite{L} to estimate their partial sums. In particular, we have
\begin{equation}\label{eqn_perron1}
\sum_{\substack{x \in \SL_2(\Z) \backslash V_{\Z} \\ \pm \Disc(x) < X}} \frac{1}{|\Stab(x)|}
= \int_{2 - i \infty}^{2 + i \infty} \xi^{\pm}(s) \frac{X^s}{s} ds
= \Res_{s = 1} \xi^{\pm}(s) X + \frac{6}{5} \Res_{s = 5/6} \xi^{\pm}(s) X^{5/6} + O(X^{3/5 + \epsilon}).
\end{equation}
One may also translate this into an estimate for the number of cubic {\itshape orders} of discriminant at most $X$,
with main terms of order $X$ and $X^{5/6}$ and error $\ll X^{3/5 + \epsilon}$.

To count cubic fields, we introduce the {\itshape $q$-nonmaximal Shintani zeta function} $\xi^{\pm}_q(s)$,
which counts only those cubic forms in \eqref{eqn_shintani_def_1} which correspond to orders which are nonmaximal at
$q$.
By work of Datskovsky and Wright \cite{DW1, DW2} or F. Sato \cite{sato}, it follows that
these zeta functions have analytic continuation and functional equations, so that their partial sums
may be estimated as in \eqref{eqn_perron1}. These partial sums are closely related to $N^{\pm}(q, X)$,
and we incorporate estimates for these sums into the sieve \eqref{eqn_BBP} to obtain our results.

The main technical difficulty is that 
the error terms in \eqref{eqn_perron1} now depend on $q$, and we must explicitly analyze this dependence.
The key step is a careful analysis of certain {\itshape cubic Gauss sums}
appearing in the functional equations for the $q$-nonmaximal Shintani zeta functions. This is carried out in
our companion paper \cite{TT}. These Gauss sums are small on average, so we obtain
error terms in \eqref{eqn_perron1} which are not too bad in $q$-aspect. This fact allows us to take a large cutoff
$Q$ in \eqref{eqn_BBP} and obtain a reasonably good error
term in Roberts' conjecture.

We will in fact introduce a smoothing technique to obtain better error terms, but this discussion illustrates the principle of our
proof. 
\\
\\
{\bf Notation.} For the most part our choice of notation is standard. Throughout, $p$ will denote a prime and $q$ a
squarefree integer. We have referred to $\xi^{\pm}(s)$ as
``Shintani zeta functions'', which has some precedent in the literature but is not universal. We also
remark that the notation $\xi_1(s)$ and $\xi_2(s)$ seems to be common in place of $\xi^{\pm}(s)$, but we
did not want to risk confusion with the numerical parameter $q$.

Throughout, $\epsilon$ will denote a positive number which may be taken to be arbitrarily small, 
not necessarily the same at each occurrence. Any implied constants will always be allowed to
depend on $\epsilon$. 

As is familiar in analytic number theory, we write $\omega(n)$ and $\Omega(n)$ for the
number of prime divisors of a positive integer $n$, counted without and with multiplicity respectively. 
It is not difficult to prove that $\omega(n)$ satisfies the bound $A^{\omega(n)} \ll_{A, \epsilon} n^{\epsilon}$
for any $A > 1$, and we will use this bound frequently.
\\
\\
{\bf Remark.} At this time, our companion paper \cite{TT} is still in preparation, and a preliminary version
may be found on the first 
author's website\footnote{\url{http://www.math.kobe-u.ac.jp/~tani/}}. After \cite{TT} is finished, we will update the references in this 
paper.

We find it convenient to refer to \cite{TT} for facts about Shintani zeta functions, but we emphasize that
\cite{TT} builds on the work of other mathematicians, especially Datskovsky and Wright
\cite{DW1, DW2, DW3}. Some of the results quoted from \cite{TT} are originally due to Datskovsky and Wright
and appear in \cite{DW1, DW2}. Please see our companion paper for
a more specific discussion of where our work builds upon that of Datskovsky and Wright.
\\
\\
{\bf Organization of this paper.} On account of the excellent exposition in Bhargava's paper 
(\cite{B}; see also \cite{BST}), we will not say any more about the Davenport-Heilbronn and Delone-Faddeev
correspondences. Instead, we begin in Section \ref{sec_shintani} with the analytic theory of the $q$-nonmaximal zeta functions.
This theory is developed in \cite{TT} and we summarize it here. We also describe the original heuristic argument of Datskovsky-Wright and Roberts, which is
quite similar to our proof. 

In addition, we discuss recent and ongoing work and some open problems at the end of Section \ref{sec_shintani}. We postponed this
discussion from the introduction so we could use the language developed in Section \ref{sec_shintani}.

In Section \ref{section_dual_bounds} we prove bounds for certain partial sums of the duals to the $q$-nonmaximal Shintani zeta
functions. These will be needed in Section \ref{sec_proof}, and their proofs use
corresponding bounds on the cubic Gauss sums, proved in \cite{TT}.

In Section \ref{section_torsion_bounds} we carry out the analysis in Section \ref{section_dual_bounds}
for the 3-torsion problem. (This section may be skipped by readers only interested in the proof of Roberts' 
conjecture.) This is the only part of the proof that is substantially different for this problem, and
indeed a new technical difficulty appears which must be resolved.

In Section \ref{sec_proof} we present our proof of Roberts' conjecture.
We first reduce Roberts' conjecture to a statement about
partial sums of Shintani zeta functions. We then estimate these 
sums using a variation of \eqref{eqn_perron1}, due essentially to Chandrasekharan and Narasimhan \cite{CN}. 

We conclude in Section \ref{sec_generalizations} by proving our more general results, including the extension to
3-torsion in quadratic fields. As we describe, almost all of our arguments
carry over to the general case.

\section*{Acknowledgments}
To be added later, after the referee report is received.

\section{$q$-nonmaximal Shintani zeta functions and their duals}\label{sec_shintani}
We recall that the {\itshape Shintani zeta functions associated to the space of binary cubic forms} are defined by the Dirichlet
series
\begin{equation}\label{eqn_def_shintani}
\xi^{\pm}(s) := \sum_{x \in \SL_2(\Z) \backslash V_{\Z}} \frac{1}{|\Stab(x)|} |\Disc(x)|^{-s},
\end{equation}
where $V_{\Z}$ was
defined in \eqref{def_vz}, the sum ranges over points of positive or negative discriminant respectively, and $\Stab(x)$, 
the stabilizer of a point $x$ in $\SL_2(\Z)$, is an abelian group of order 1 or 3. We now introduce {\itshape $q$-nonmaximal}
analogues of these zeta functions. Throughout, $q$ will be a squarefree integer.

\begin{definition}\label{def_q_nonmax}
The {\upshape $q$-nonmaximal
Shintani zeta functions} $\xi^{\pm}_q(s)$ are defined by the formula \eqref{eqn_def_shintani}, with the sum restricted to 
those $x$ not belonging
to $U_p$ (defined in Proposition \ref{def_up}) for any $p | q$.
\end{definition}

In this section we describe the analytic theory of these
functions, following Shintani \cite{S}, Datskovsky-Wright \cite{DW1, DW2}, and our companion paper
\cite{TT}.

Shintani's original theorem relates the zeta functions $\xi^{\pm}(s)$ to dual zeta functions
\begin{equation}\label{eqn_def_dual_shintani}
\widehat{\xi}^{\pm}(s) := \sum_{x \in \SL_2(\Z) \backslash \widehat{V}_{\Z}} \frac{1}{|\Stab(x)|} |\Disc(x)|^{-s},
\end{equation}
where $\widehat{V}_{\Z}$, the dual lattice to $V_{\Z}$, is defined by
\begin{equation}
\widehat{V}_{\Z} := \{ a u^3 + b u^2 v + c u v^2 + d v^3 \ : a, d \in \Z; \ \  b, c \in 3 \Z \}.
\end{equation}

To describe the analogue for the $q$-nonmaximal zeta functions $\xi^{\pm}_q(s)$, we must introduce the cubic Gauss sums.
We define $\Phi_q(x)$ to be the characteristic
function of those $x$ not in $U_p$ for any $p | q$, defined on either $V_{\Z}$ or $V_{\Z / q^2 \Z}$. 
The {\itshape cubic Gauss sum} is the dual to $\Phi_q(x)$, defined as
the following function on $\widehat{V}_{\Z}$:
\begin{equation}\label{def_phihat}
\widehat{\Phi}_q(x) := \frac{1}{q^8} \sum_{y \in V_{\Z / q^2 \Z}} \Phi_q(y) \exp(2 \pi i [x, y] / q^2).
\end{equation}
Here
\begin{equation}
[x, y] = x_4 y_1 - \frac{1}{3} x_3 y_2 + \frac{1}{3} x_2 y_3 - x_1 y_4
\end{equation}
is the alternating bilinear form used to identify $V$ with $\widehat{V}$, where
$x_i$ and $y_j$ are the coordinates of $x$ and $y$ respectively.

\begin{remark}
The dual $\widehat{\Phi}_q(x)$ might be thought of as a sum over $\frac{1}{q^2} \Z / \Z$, as it arises as
a product of $p$-adic Fourier transforms of the function $\Phi_q$. This integral reduces naturally to a finite
sum over  $V_{\frac{1}{q^2} \Z / \Z}$, which is equivalent to the sum given above.
\end{remark}

We observe that  $\widehat{\Phi}_q$ satisfies the multiplicativity property
\begin{equation}\label{eqn_mult}
\widehat{\Phi}_q(x) \widehat{\Phi}_{q'}(x) = \widehat{\Phi}_{q q'}(x)
\end{equation}
for all $(q, q') = 1$. We also note that if $x$ corresponds to a cubic ring $R$ under the Delone-Faddeev
correspondence, then $\widehat{\Phi}_p(x)$ depends only on $R \otimes_{\Z} \Z_p$. This implies that
$\widehat{\Phi}_p(x) = \widehat{\Phi}_p(x')$ if $x$ and $x'$ correspond to $R$ and $R'$, where $R'$ is
contained in $R$ with index coprime to $p$.

We are now prepared to describe the analytic properties of $\xi^{\pm}_q(s)$.

\begin{theorem}[Shintani \cite{S}; Datskovsky and Wright \cite{DW2}; \cite{TT}]\label{thm_DWT}
The $q$-nonmaximal Shintani zeta functions $\xi^{\pm}_q(s)$ converge absolutely for $\Re(s) > 1$, have analytic continuation to
all of $\C$, holomorphic except for poles at $s = 1$ and $s = 5/6$, and satisfy the functional equation
\begin{equation}\label{eqn_shintani_FE}
\left(
\begin{array}{c}
\xi^+_q(1 - s) \cr
\xi^-_q(1 - s) \cr
\end{array}
\right)
= 
\Gamma\bigg(s - \frac{1}{6}\bigg) \Gamma(s)^2 \Gamma\bigg(s + \frac{1}{6}\bigg) 
2^{-1} 3^{6s - 2} \pi^{-4s} 
\left(
\begin{array}{cc}
\sin 2 \pi s & \sin \pi s \cr
3 \sin \pi s & \sin 2 \pi s \cr
\end{array}
\right)
\left(
\begin{array}{c}
\widehat{\xi}^+_q(s) \cr
\widehat{\xi}^-_q(s) \cr
\end{array}
\right)
,
\end{equation}
where the dual $q$-nonmaximal Shintani zeta functions are given by
\begin{equation}\label{def_shintani_dual}
\widehat{\xi}^{\pm}_q(s) := \sum_{x \in \SL_2(\Z) \backslash \widehat{V}_{\Z}} \frac{1}{|\Stab(x)|}
\widehat{\Phi}_q(x) \big( |\Disc(x)| / q^8 \big)^{-s}.
\end{equation}

The residues are given by
\begin{equation}\label{eqn_residue}
\Res_{s = 1} \xi^{\pm}_q(s) = 
\alpha^{\pm} \prod_{p | q} \bigg(
\frac{1}{p^2} + \frac{1}{p^3} - \frac{1}{p^5} \bigg)
+
\beta \prod_{p | q} \bigg(
\frac{2}{p^2} - \frac{1}{p^4} \bigg),
\end{equation}
and
\begin{equation}\label{eqn_residue_56}
\Res_{s = 5/6} \xi^{\pm}_q(s) = \gamma^{\pm} \prod_{p | q} \bigg(
\frac{1}{p^{5/3}} + \frac{1}{p^2} - \frac{1}{p^{11/3}} \bigg),
\end{equation}
where $\alpha^+ = \pi^2 / 36$, $\alpha^- = \pi^2 / 12$, $\beta = \pi^2 / 12$,
$\gamma^+ = \frac{\zeta(1/3) \Gamma(1/3)^3}{4 \sqrt{3} \pi}$, and $\gamma^- = \sqrt{3} \gamma^+.$
\end{theorem}
We introduce the notation
$$\widehat{\xi}^{\pm}_q(s) =: \sum_{\mu_n} b^{\pm}_q(\mu_n) {\mu_n}^{-s}$$
for the dual Shintani zeta function, where $\mu_n \in \frac{1}{q^8} \Z$ refers to the quantity
$|\Disc(x)|/q^8$ in \eqref{def_shintani_dual}. Note that in \cite{TT} the factor
of $q^{8s}$ appears in \eqref{eqn_shintani_FE} instead of \eqref{def_shintani_dual}. We chose the
normalization here to get a uniform shape for the functional
equation for all $q$, and to be consistent
with our analytic reference \cite{CN}, which we will describe later.

We can now present the heuristic argument of
Datskovsky-Wright and Roberts.
For any set of primes $\mathcal{P}$, let $N^{\pm}_{3, \mathcal{P}}(X)$ denote
the number of cubic orders $\mathcal{O}$ with $\pm \Disc(\mathcal{O}) < X$ which are maximal at all primes in $\mathcal{P}$.
Assuming\footnote{In fact our assumption is a bit rash, but for $\mathcal{P} = \emptyset$ see \cite{shin_reducible} for a proof with an error term
of $O(X^{2/3 + \epsilon})$.}
that we can separately estimate and subtract the contribution from reducible rings, the equations above imply that
\begin{equation}\label{eqn_roberts_heuristic}
N^{\pm}_{3, \mathcal{P}}(X) = 
\frac{1}{2} \alpha^{\pm} X \prod_{p \in \mathcal{P}} \bigg(1 - \frac{1}{p^2}\bigg) \bigg(1 - \frac{1}{p^3}\bigg)
+ \frac{3}{5} \gamma^{\pm} X^{5/6} \prod_{p \in \mathcal{P}}
\bigg(1 - \frac{1}{p^{5/3}} \bigg) \bigg(1 - \frac{1}{p^2} \bigg) + O_{\mathcal{P}}(X^{3/5 + \epsilon}).
\end{equation}
Formally taking a limit as $\mathcal{P}$ tends to the set of all primes, we obtain \eqref{conj_rc}.

We also see that the inclusion-exclusion sieve \eqref{eqn_BBP} introduced by Belabas, Bhargava, 
and Pomerance allows us {\itshape in principle} to prove Roberts' conjecture.
However, without an analysis of the $\mathcal{P}$-dependence of the error term in \eqref{eqn_roberts_heuristic},
it is unclear that we can obtain an error term smaller than $X^{5/6}$. Indeed, our initial
attempts yielded ``proofs'' of Roberts' conjecture with error terms that were too large.

To analyze the $\mathcal{P}$-dependence of the error terms in \eqref{eqn_roberts_heuristic},
we must study the cubic Gauss sum in \eqref{def_phihat}.This sum is studied in \cite{TT}, and Section \ref{section_dual_bounds}
we will use these results to prove bounds on appropriate partial sums of the dual zeta functions $\widehat{\xi}^{\pm}_q(s)$.

Before proceeding, we will simplify the functional equations by using a diagonalization argument of Datskovsky and Wright.
At the end of this section, 
we will also describe some related work involving Shintani zeta functions.

\subsection{Datskovsky and Wright's diagonalization}\label{subsec_diag} To simplify our analysis
we apply an observation of Datskovsky and Wright \cite{DW2}.
The functional equations above have a curious matrix form, such that the negative and positive discriminant
Shintani zeta functions are interdependent. By diagonalizing this matrix, we can
greatly simplify the form of the functional equation. 

We define, for each $q$, diagonalized Shintani zeta functions
\begin{equation}\label{eqn_diag_1}
\xi^{\add}_q(s) := 3^{1/2} \xi^+_q(s) + \xi^-_q(s),
\end{equation}
\begin{equation}\label{eqn_diag_2}
\xi^{\sub}_q(s) := 3^{1/2} \xi^+_q(s) - \xi^-_q(s).
\end{equation}
We diagonalize the dual zeta functions in exactly the same way.

The diagonalizations then take the following shape. Define
$$\Lambda^{\add}_q(s) := \bigg( \frac{2^4 \cdot 3^6}{\pi^4} \bigg)^{s/2} \Gamma\bigg(\frac{s}{2}\bigg)
 \Gamma\bigg(\frac{s}{2} + \frac{1}{2} \bigg) \Gamma\bigg(\frac{s}{2} + \frac{1}{12} \bigg) \Gamma\bigg(\frac{s}{2} - \frac{1}{12}\bigg)
\xi^{\add}_q(s),$$
$$\Lambda^{\sub}_q(s) := \bigg( \frac{2^4 \cdot 3^6}{\pi^4} \bigg)^{s/2} \Gamma\bigg(\frac{s}{2}\bigg)
 \Gamma\bigg(\frac{s}{2} + \frac{1}{2} \bigg) \Gamma\bigg(\frac{s}{2} + \frac{5}{12} \bigg) \Gamma\bigg(\frac{s}{2} + \frac{7}{12}\bigg)
\xi^{\sub}_q(s),$$
and define $\widehat{\Lambda}^{\add}_q(s)$ and $\widehat{\Lambda}^{\sub}_q(s)$ in the same way. Then, the functional equations take the
shape 
\begin{equation}\label{eqn_diag_fe}
\Lambda^{\add}_q(1 - s) =  3 \widehat{\Lambda}^{\add}_q(s),
\end{equation}
\begin{equation}\label{eqn_diag_fe2}
\Lambda^{\sub}_q(1 - s) =  -3 \widehat{\Lambda}^{\sub}_q(s).
\end{equation}
This is the classical shape for functional equations of zeta functions, apart from the interesting factors of $\pm 3$, 
and it will be a convenient one to work with.
We also note the interesting fact that only $\Lambda_{\add}$, and not $\Lambda_{\sub}$, retains the pole at
$s = 5/6$.

\subsection{Some related work} We conclude this section by describing some recent and ongoing related work.
These results will not be needed elsewhere in this paper, but we hope that they may prove useful in addressing
related problems.

We first mention a striking result, conjectured by Ohno \cite{O} and then proved by Nakagawa \cite{N}. They established
that the dual Shintani zeta functions are related to the original Shintani zeta functions by the simple formulas
\begin{equation}
\widehat{\xi}^+(s) = 3^{-3s} \xi^-(s),
\end{equation}
\begin{equation}
\widehat{\xi}^-(s) = 3^{1 - 3s} \xi^+(s).
\end{equation}
One can incorporate these formulas into Datskovsky and Wright's diagonalization, and therefore put the classical 
Shintani zeta functions into a self-dual form, with functional equations related to the ones above.

More recently, Ohno, the first author, and Wakatsuki \cite{OTW, OT} classified all of the $\SL_2(\Z)$-invariant sublattices
of $V_{\Z}$, and proved that the Shintani zeta functions associated to these lattices share the nice properties above.

There is also the work of Yukie \cite{Y}, who has initiated the study of {\itshape quartic} Shintani zeta functions,
which are associated to a certain 12-dimensional prehomogeneous vector space. These zeta functions have not yet
been studied as thoroughly as their cubic analogues, but it seems that one may be able to prove estimates for quartic
fields with power saving error terms, and perhaps improve the result of Belabas, Bhargava, and Pomerance \cite{B_quartic, BBP}. Moreover, if any
secondary terms are present, this method seems likely to yield them, at least in principle. However, this approach comes
with substantial technical difficulties, and so far it has yet to even yield the main term. 

We may also study extensions of base fields other than $\Q$. In \cite{DW3}, Datskovsky and Wright proved the analogue
of the Davenport-Heilbronn theorem for any global field of characteristic not equal to 2 or 3. They also suggest
that secondary terms should appear in this case as well. Moreover, Morra \cite{M} has designed and implemented an algorithm
to compute cubic extensions of imaginary quadratic fields of class number 1. At present we have verified that Morra's
calculations closely match the Datskovsky-Wright heuristics for extensions of $\Q(i)$.

In principle we expect to be able to prove an analogue of Roberts' conjecture in this general setting. However, 
we expect that our error terms would be larger than $X^{5/6}$, even for cubic extensions of quadratic fields. However, one may be able to establish secondary
terms for {\itshape smoothed} sums, such as
$$\sum_K |\Disc(K)| \exp^{-|\Disc(K)| / X},$$
where $K$ ranges over cubic extensions of a fixed number field. We look forward to investigating this in the near future.

There is also a much more general theory of prehomogeneous vector spaces and their zeta functions, developed in the seminal
works of Sato-Kimura \cite{SK}, Sato-Shintani \cite{SS}, and Wright-Yukie \cite{WY}, among many others. Many authors have applied this
theory to obtain a variety of interesting arithmetic density results, and it is possible that the methods of this paper might be applied
to further refine some of these results. As one example we mention work of the first author \cite{tani_simple}, studying
the zeta functions associated to some non-split forms of representations of $\GL_2(k) \times \GL_n(k)^2$ on the space $k^2 \otimes k^n \otimes k^n$
for $n = 2$ or $3$. These zeta function are proved (\cite{tani_simple}, Theorem 4.24) to be meromorphic with two simple poles. In the case
$n = 2$, this yields an arithmetic density result for the average size of $(h_F R_F)^2$, where $F$ ranges over quadratic extensions
of $k$, and $h_F$ and $R_F$ denote the class number and regulator respectively. For $n = 3$, this work is incomplete, but similar
methods should yield a result for the average size of $h_F R_F$, where $F$ now ranges over cubic extensions of $k$. Moreover,
in the cubic case it can be shown that 
the secondary pole of the zeta function does not vanish when twisted by cubic characters. This suggests that an analogue of Theorem
\ref{thm_progressions} might hold for this case as well.

Finally, the methods of this paper may be used to prove statements about prime and almost-prime discriminants of cubic
fields. When one replaces the $q$-nonmaximality condition with a divisibility condition on the discriminant,
the methods of this paper yield estimates for the number of discriminants divisible by $q$, and combining these estimates with
different sieve methods allows us to prove a variety of results. However, we were unable to improve upon results of Belabas
and Fouvry \cite{BF}, and so we did not pursue this further.

\section{Bounds for duals of the $q$-nonmaximal Shintani zeta function}\label{section_dual_bounds}

Let 
$$\widehat{\xi}^{\pm}_q(s) =: \sum_{\mu_n} b^{\pm}_q(\mu_n) {\mu_n}^{-s}$$
be the dual $q$-nonmaximal Shintani zeta functions, defined in \eqref{def_shintani_dual}.
Throughout, we will fix a choice of sign and drop the $\pm$ from our notation. We also recall
that the sum is over $\mu_n \in \frac{1}{q^8} \Z$.

Our later analytic estimates will require bounds for partial sums of the $b_q(\mu_n)$.
The primary goal of this section will be to prove the following bound.

\begin{theorem}\label{thm_trivial_bound}
We have the bound
\begin{equation}\label{eqn_trivial_bound}
\sum_{\mu_n < X} |b_q(\mu_n)| \ll q^{1 + \epsilon} X,
\end{equation}
uniformly for all $q$ and $X$.
\end{theorem}

The proof essentially involves two steps. The first is an analysis of the Gauss sums $\widehat{\Phi}_q(x)$, carried out in
\cite{TT}. Our 
analysis (see Lemma \ref{lem_gauss_sum}) shows that the Gauss sums are only supported on certain $\GL_2(\Z / q^2 \Z)$-orbits of
$V_{\Z / q^2 \Z}$, and in particular that cubic rings contributing to \eqref{eqn_trivial_bound} must be either nonmaximal
or totally ramified at each prime dividing $q$. 

In the second step, we use a counting argument to bound the contribution of each orbit type, largely
following work of
Belabas, Bhargava, and Pomerance \cite{BBP}.

Before presenting the details, we derive the bounds that we will need later.

\begin{proposition}\label{prop_high_middle_bound}
For any $z$ and any $\delta > 1$, we have the bound
\begin{equation}\label{eqn_high_bound}
\sum_{\mu_n > z} |b_q(\mu_n)| {\mu_n}^{-\delta} \ll_{\delta} q^{1 + \epsilon} z^{-\delta + 1}.
\end{equation}
Furthermore, for any $\delta \in (0, 1)$, we have the bound
\begin{equation}
\sum_{\mu_n < z} |b_q(\mu_n)| {\mu_n}^{-\delta} \ll_{\delta} q^{1 + \epsilon} z^{-\delta + 1}.
\end{equation}
Both bounds are uniform in $q$.
\end{proposition}

\begin{proof}
To prove these bounds, we divide the respective intervals into dyadic subintervals of the form $[y, 2y]$.
By \eqref{eqn_trivial_bound}, the contribution of each such interval is $\ll q^{1 + \epsilon} y^{-\delta + 1}$.
Both bounds now follow by summing over $y$.
\end{proof}

To prepare for the proof of Theorem \ref{thm_trivial_bound}, we restate \eqref{eqn_trivial_bound}
in the form
\begin{equation}\label{eqn_to_be_proved}
\sum_{|\Disc(x)| < Y} |\widehat{\Phi}_q(x)| \ll q^{-7 + \epsilon} Y,
\end{equation}
where the sum is over integral binary cubic forms up to $\GL_2(\Z)$-equivalence. In light of the Delone-Faddeev
correspondence, we may (and do) refer to the $x$ as either cubic forms or cubic rings. We will find it convenient
to talk about divisibility (i.e., content) in terms of forms, and maximality properties in terms of rings.

Exact formulas for $\widehat{\Phi}_q(x)$ are proved in \cite{TT}, and the following lemma extracts the results we need:
\begin{lemma}\label{lem_gauss_sum}\cite{TT}
The function $\widehat{\Phi}_q(x)$ is multiplicative in $q$. Moreover, for a prime $p > 3$, 
the value of $\widehat{\Phi}_p(x)$ is given by the following table (where $R$ is the cubic ring corresponding
to $x$):
\begin{itemize}
\item {\upshape (Content $p^2$:)} $\widehat{\Phi}_p(x) = p^{-2} + p^{-3} - p^{-5}$ if $p^2$ divides the content of $R$.
\item {\upshape (Content $p$:)} $|\widehat{\Phi}_p(x)| < p^{-3}$ if $p$ divides the content of $R$, but $p^2$ does not.
\item {\upshape (Divisible by $p^4$:)} $\widehat{\Phi}_p(x) = p^{-3} - p^{-5}$ if $R$ is nonmaximal at $p$, has content
coprime to $p$, and $p^4 | \Disc(R)$.
\item {\upshape (Divisible by $p^2$:)} $|\widehat{\Phi}_p(x)| = p^{-5}$ for certain other rings for which $p^2 | \Disc(R)$. (In particular,
whenever $R \otimes_{\Z} \Q$ is totally ramified at $p$ and $R$ does not belong to any of the first three
categories.)
\item Otherwise, and in particular if
$p^2 \nmid \Disc(R)$, we have
  $\widehat{\Phi}_p(x) = 0$.
\end{itemize}
\end{lemma}
A few remarks are in order. We have excluded $p = 2, 3$, but for these two primes 
we may incorporate the trivial estimate 
$|\widehat{\Phi}_p(x)| \leq 1$ into our implied constants. Also, we recall \eqref{eqn_mult} and the note
afterwards, which will be used in our proof. Finally, we note that
any $R$ with discriminant divisible by $p^4$ ($p > 3$) is in fact automatically nonmaximal at $p$, so that
there is some redundancy in the conditions described above.

The basic idea of the proof is to separate the contributions to \eqref{eqn_trivial_bound} according to
the list above, and then count the number of each type of ring. If we could prove that the number of cubic
rings $R$ for which $d | \Disc(R)$ and $|\Disc(R)| < X$ was $\ll X/d^{1 - \epsilon}$, uniformly
in $X$ and $d$, then the theorem would quickly follow. This seems to be difficult in general, but
we will be able to prove an adequate substitute. 
We begin with the case where $d = r^2$ for squarefree $r$, where
Belabas, Bhargava, and Pomerance \cite{BBP} proved the inequality described above:
\begin{lemma}\label{lemma_BBP}
For squarefree $r$, we have the bound
\begin{equation}\label{eqn_to_be_proved_base}
\sum_{\substack{|\Disc(x)| < Y \\ r^2 | \Disc(x)}} 1 < M 6^{\omega(r)} Y/r^2,
\end{equation}
for an absolute constant $M$.
\end{lemma}

\begin{proof} For those $x$ corresponding to cubic orders, this is Lemma 3.4 of \cite{BBP}, and 
for general $x$ it may be proved in the same manner. We briefly recall the details.\footnote{We
refer to the published version of \cite{BBP}, which offers a different proof than some preprints of \cite{BBP}.} 

For any factorization $r = ab$, we count the number
of maximal cubic rings $R$ with discriminant divisible by $b^2$, and then the number
of cubic rings $R'$ contained in $R$ with index divisible by $a$. We then obtain
\eqref{eqn_to_be_proved_base} by summing over all such factorizations.

The count of maximal cubic rings is $\ll Y 3^{\omega(b)}/b^2$ by Lemma 3.3 of \cite{BBP} in the irreducible
case, and the same bound follows trivially for the reducible case 
(as there is at most one such ring of any given
discriminant).

Now for any fixed maximal cubic ring $R$, write
$\eta_{R}$ for the generating series counting the subrings of index $n$ contained in $R$.
By work of Datskovsky-Wright (\cite{DW2}, Theorem 6.1), we have the coefficientwise bound
\begin{equation}
\eta_{R}(s) \preceq \zeta(2s) \zeta(3s - 1) \zeta(s)^3.
\end{equation}
Arguing exactly as in \cite{BBP}, the number of cubic rings being counted is
\begin{equation}
\ll Y \frac{3^{\omega(ab)}}{(ab)^2} \sum_{j \geq 1} \frac{3^{\Omega(j)}}{j^2}
\sum_{v, w: \ v w^2 | j} w.
\end{equation}
The sum over $j$ converges, and one concludes the argument by
summing over the $2^{\omega(r)}$ choices of $a$ and $b$.
\end{proof}
We would like an analogue of \eqref{eqn_to_be_proved_base} where $r$ is not required to be squarefree.
The methods of \cite{BBP} do not extend to this case, but 
Lemma \ref{lem_gauss_sum} gives us the additional information that any such rings occurring in
\eqref{eqn_to_be_proved}
are nonmaximal. This allows us to apply the following lemma:

\begin{lemma}\label{lemma_overrings}
Let $R$ be a cubic ring which is nonmaximal at $p$, and whose content is coprime to $p$. Then
$R$ is contained in an overring $R'$ with index $p$.
There are at most $3$ rings $R$ so contained in any $R'$ whose content is coprime to $p$,
and $p + 1$ rings otherwise.
\end{lemma}
\begin{proof}
Suppose that $x = (a, b, c, d) \in V_{\Z}$ is an element of the $\GL_2(\Z)$-orbit corresponding to $R$. By
Proposition \ref{def_up}, we may assume that $p^2 | a$ and $p | b$. Then, by the explicit form of the
Delone-Faddeev correspondence (see, e.g., \cite{GGS}) we can write $R$
as $R = \Z 1 \times \Z \omega \times \Z \theta$, where
\begin{equation}
\omega^2 = -ac - b\omega + a \theta, \ \
\theta^2 = - bd - d \omega + c \theta, \ \
\omega \theta = - ad.
\end{equation}
Let $R' := \Z 1 \times \Z (\omega/p) \times \Z \theta \subseteq R \otimes_{\Z} \Q$. Then, $R'$ is closed under multiplication because
\begin{equation}
(\omega/p)^2 = -(a/p^2)c - (b/p)(\omega/p) + (a/p^2)\theta, \ \ 
\theta^2 = -bd - dp(\omega/p) + c\theta, \ \
(\omega/p)\theta = -(a/p)d.
\end{equation}
It follows that $R'$ is a cubic ring which contains $R$ with index $p$, and that both are 
contained in the same maximal ring. 

The bound on the number of such $R$ is proved in 
Lemma 2.4 of \cite{BBP}; although this lemma is stated for cubic orders, its proof remains valid for any cubic ring.
\end{proof}

\begin{remark} In terms of cubic forms, this construction sends $(p^3 a, p^2 b, p c, d)$ to $(pa, pb, pc, pd)$,
which illustrates that a ring with trivial content can be contained in a ring with content $p$.
\end{remark}

\begin{proof}[Proof of Theorem \ref{thm_trivial_bound}] 
For each factorization $q = abcd$, consider the contribution to
\eqref{eqn_to_be_proved} from those $x$ satisfying the following:
\begin{itemize}
\item
If $p| a$, then $p^2|\ctt(x)$;
\item
if $p| b$, then $p^2\nmid\ctt(x)$ but $p|\ctt(x)$;
\item
if $p| c$, then $p\nmid\ctt(x)$ but $p^4|\Disc(x)$;
\item
if $p| d$, then $p^4\nmid\Disc(x)$ but $p^2|\Disc(x)$.
\end{itemize}
(Here $\ctt(x)$ denotes the content of $x$.) Lemma \ref{lem_gauss_sum} implies that for each $x$ we have
\begin{equation}
|\widehat{\Phi}_q(x)| \leq \frac{2^{\omega(a)}}{a^2 b^3 c^3 d^5}.
\end{equation}
We use Lemma \ref{lemma_overrings} to replace each $x$ by an overring $x'$
of index $c$. For each $x'$, we define $c' | c$ by $c' = \gcd(\ctt(x'), c)$, such that there are at
most $3^{\omega(c)} c'$ rings $x$ corresponding to each overring $x'$. We note also that the discriminant of
each $x'$ is divisible by $(\frac{cd}{c'})^2$. (It is also divisible by factors which divide the content.)

For each choice of $a, b, c, c', d$, the contribution to \eqref{eqn_to_be_proved} is therefore
\begin{equation}
\leq \frac{2^{\omega(a)} 3^{\omega(c)} c'}{a^2 b^3 c^3 d^5}
\sum_{\substack{|\Disc(x')| < Y/c^2 \\ a^2 b c' | \ctt(x') \\ ( \frac{cd}{c'} )^2 | \Disc(x')}} 1
= 
\frac{2^{\omega(a)} 3^{\omega(c)} c'}{a^2 b^3 c^3 d^5}
\sum_{\substack{|\Disc(x')| < \frac{Y}{a^8 b^4 c'^4 c^2} \\  ( \frac{cd}{c'} )^2 | \Disc(x')}} 1.
\end{equation}
By Lemma \ref{lemma_BBP}, this is
\begin{equation}
\ll
\frac{2^{\omega(a)} 3^{\omega(c)} c'}{a^2 b^3 c^3 d^5}
\cdot 6^{\omega(cd)} \frac{Y}{a^8 b^4 c'^2 c^4 d^2}
\ll 18^{\omega(q)} \frac{Y}{a^{10} b^7 c^7 c' d^7 } \ll q^{-7 + \epsilon} Y.
\end{equation}
The theorem follows by summing over the $5^{\omega(q)} \ll q^{\epsilon}$ factorizations $q = abcd$ and choices for $c'$.
\end{proof}

\section{Bounds for the dual Shintani zeta function in the 3-torsion problem}\label{section_torsion_bounds}

In this section we will carry out the analysis of Section \ref{section_dual_bounds} for the related problem of estimating
3-torsion in class groups.
In particular, throught this section, $\Phi_p(x)$ and $b_q(\mu_n)$ will correspond to the (complement of the) 
set $V_p$ instead of $U_p$.
This set was defined in \cite{DH}, and we recall the definition in Section \ref{sec_generalizations}. 

The idea of the proof is very much the same, but one new technical difficulty appears: The Fourier transform $\widehat{\Phi}_p(x)$
will take a form which is more difficult to estimate over $|\Disc(x)| < Y$ when $p$ is large in relation to $Y$. As a result,
we will be limited to proving the following
analogue of Theorem \ref{thm_trivial_bound}:
\begin{theorem}\label{thm_trivial_bound_2}
If $\Phi_p(x)$ corresponds to the complement of $V_p$, then we have the bounds
\begin{equation}\label{eqn_trivial_bound_2}
\sum_{\mu_n < X} |b_q(\mu_n)| \ll q^{2 + \epsilon} X
\end{equation}
and
\begin{equation}\label{eqn_trivial_bound_3}
\sum_{\mu_n < X} |b_q(\mu_n)| \ll q^{1 + \epsilon} X + q^{-1 + \epsilon},
\end{equation}
uniformly for all $q$ and $X$.
\end{theorem}

\begin{remark}
The bound \eqref{eqn_trivial_bound_2} is quite simple to prove (assuming the results
of the previous section), and as we show at the end of Section \ref{subsec_torsion}, this already suffices to obtain Theorem
\ref{thm_rc_torsion} with a larger error term of $O(X^{9/11 + \epsilon})$. This section
describes a ``trick'' which allows us to obtain \eqref{eqn_trivial_bound_3} and thus
an error term of $O(X^{18/23 + \epsilon})$, and may be skipped without loss of continuity.
\end{remark}

We obtain the following corollary in the same way as before.
\begin{proposition}\label{prop_high_middle_bound2}
For $\delta \in (0, 1)$, we have the bounds
\begin{equation}
\sum_{\mu_n < z} |b_q(\mu_n)| {\mu_n}^{-\delta} \ll_{\delta} q^{2 + \epsilon} z^{-\delta + 1},
\end{equation}
when $z \leq q^{-3}$, and
\begin{equation}
\sum_{q^{-3} < \mu_n < q^{-2}} |b_q(\mu_n)| {\mu_n}^{-\delta} \ll_{\delta} q^{3 \delta - 1 + \epsilon},
\end{equation}
\begin{equation}
\sum_{q^{-2} < \mu_n < z} |b_q(\mu_n)| {\mu_n}^{-\delta} \ll_{\delta} q^{1 + \epsilon} z^{-\delta + 1},
\end{equation}
when $z > q^{-2}.$ We also obtain, as before, for any $\delta > 1$ and any $z > q^{-2}$,
\begin{equation}\label{eqn_high_bound_2}
\sum_{\mu_n > z} |b_q(\mu_n)| {\mu_n}^{-\delta} \ll_{\delta} q^{1 + \epsilon} z^{-\delta + 1}.
\end{equation}
\end{proposition}
We now come to the proof of Theorem \ref{thm_trivial_bound_2}. We begin with the following analogue of
Lemma \ref{lem_gauss_sum}:

\begin{lemma}\label{lem_gauss_sum2}\cite{TT}
The function $\widehat{\Phi}_q(x)$ (now corresponding to the sets $V_p$ for $p | q$) is multiplicative in $q$. Moreover, for a prime $p > 3$, 
the value of $\widehat{\Phi}_p(x)$ is given by the following table (where $R$ is the cubic ring corresponding
to $x$):
\begin{itemize}
\item {\upshape (Content $p^2$:)} $\widehat{\Phi}_p(x) = 2 p^{-2} - p^{-4}$ if $p^2$ divides the content of $R$.
\item {\upshape (Content $p$:)} $|\widehat{\Phi}_p(x)| < 2 p^{-3}$ if $p$ divides the content of $R$, but $p^2$ does not.
\item {\upshape (Divisible by $p^4$:)} $\widehat{\Phi}_p(x) = p^{-3} - p^{-4}$ if $R$ is nonmaximal at $p$, $p$ does not
divide the content of $R$, and $p^4 | \Disc(R)$.
\item {\upshape (Divisible by $p^3$:)} $|\widehat{\Phi}_p(x)| = p^{-4}$ for certain other rings which are nonmaximal at $p$
and for which $p^3 | \Disc(R)$.

\item Otherwise, and in particular if
$p^3 \nmid \Disc(R)$, we have
  $\widehat{\Phi}_p(x) = 0$.
\end{itemize}
\end{lemma}
We recall again the remarks after Lemma \ref{lem_gauss_sum}, and observe that 
if $p^3 | \Disc(R)$ ($p > 3$), $R$ is in fact automatically nonmaximal at $p$.

In Section \ref{section_dual_bounds}, we needed to count discriminants which were divisible by $p^2$ and 
which contributed $O(p^{-5})$ each to our final estimates. Now, we must count discriminants divisible by $p^3$ and which
contribute $O(p^{-4})$ each. We expect the total contributions to be comparable in size, but we were unable to prove this.
In particular, Lemma \ref{lemma_BBP} is proved using class field theory (see Lemma 3.3 of \cite{BBP}), 
and the proof does not carry over to this case.

We will begin by applying Lemma
\ref{lemma_overrings} to reduce to counting $p$-divisible rings. We then need to bound the number of such rings,
and we can prove such a bound using the methods of this paper! The associated Dirichlet series are ``$p$-divisible
Shintani zeta functions''\footnote{They are Shintani zeta functions if we weight each ring as described in the proof of
Proposition \ref{prop_shintani_reduction}. The weights are all between $1/3$ and $2$ and we are only seeking an $O$-estimate, 
so this technical point will not affect the proof.}, 
and so we may estimate their partial sums using contour integration. 
We thus obtain a statement (Lemma \ref{lemma_BBP2}) whose proof requires \eqref{eqn_trivial_bound_2}, but which is used
in the
proof of \eqref{eqn_trivial_bound_3}. This may seem like circular reasoning;
the reason this
``circular'' argument works is that the proof of Lemma \ref{lemma_BBP} exploits similarities in the structure of $\Phi_p(x)$ and
$\widehat{\Phi}_p(x)$.

\begin{lemma}\label{lemma_BBP2}
For $d = r s^2$ with $r$ and $s$ coprime and squarefree, we have the bound
\begin{equation}\label{eqn_to_be_proved_base2}
\sum_{\substack{|\Disc(x)| < X \\ d | \Disc(x)}} 1 \ll X/d^{1 - \epsilon} + (rs)^{2 + \epsilon}.
\end{equation}
\end{lemma}

We first use this to prove Theorem \ref{thm_trivial_bound_2}, and then we will prove Lemma \ref{lemma_BBP2}.

\begin{proof}[Proof of Theorem \ref{thm_trivial_bound_2}]
The proof of \eqref{eqn_trivial_bound_2} follows by comparing Lemma \ref{lem_gauss_sum2} to Lemma \ref{lem_gauss_sum}. For each
$x$, the bound in Lemma \ref{lem_gauss_sum2} is at most $q$ times that in Lemma \ref{lem_gauss_sum}, so
\eqref{eqn_trivial_bound_2} follows from Theorem \ref{thm_trivial_bound}.

To prove \eqref{eqn_trivial_bound_3},
we reformulate our bound in the shape
\begin{equation}\label{eqn_to_be_proved2}
\sum_{|\Disc(x)| < Y} |\widehat{\Phi}_q(x)| \ll q^{-7 + \epsilon} Y + q^{-1 + \epsilon}.
\end{equation}
As in the proof of Theorem \ref{thm_trivial_bound}, 
for each factorization $q = abcd$ we consider the contribution to
\eqref{eqn_to_be_proved} from those $x$ satisfying the following:
\begin{itemize}
\item
If $p| a$, then $p^2|\ctt(x)$;
\item
if $p| b$, then $p^2\nmid\ctt(x)$ but $p|\ctt(x)$;
\item
if $p| c$, then $p\nmid\ctt(x)$ but $p^4|\Disc(x)$;
\item
if $p| d$, then $p^4\nmid\Disc(x)$ but $p^3|\Disc(x)$.
\end{itemize}
Lemma \ref{lem_gauss_sum} implies that for each $x$ we have
\begin{equation}
|\widehat{\Phi}_q(x)| \leq \frac{2^{\omega(ab)}}{a^2 b^3 c^3 d^4}.
\end{equation}
We use Lemma \ref{lemma_overrings} to replace each $x$ by an overring $x'$
of index $cd$. For each $x'$, we define $c' | c$ and $d' | d$ by $c'd' = \gcd(\ctt(x'), cd)$, such that there are at
most $3^{\omega(cd)} c' d'$ rings $x$ corresponding to each overring $x'$. We note also that the discriminant of
each $x'$ is divisible by $(\frac{c}{c'})^2 (\frac{d}{d'})$. 

For each choice of $a, b, c, c', d, d'$, the contribution to \eqref{eqn_to_be_proved} is therefore
\begin{equation}
\leq \frac{2^{\omega(a)} 3^{\omega(cd)} c' d'}{a^2 b^3 c^3 d^4}
\sum_{\substack{|\Disc(x')| < Y/(cd)^2 \\ a^2 b c' d' | \ctt(x') \\ (\frac{c}{c'})^2 (\frac{d}{d'}) | \Disc(x')}} 1
= 
\frac{2^{\omega(a)} 3^{\omega(cd)} c' d'}{a^2 b^3 c^3 d^4}
\sum_{\substack{|\Disc(x')| < \frac{Y}{a^8 b^4 c'^4 c^2 d'^4 d^2} \\ (\frac{c}{c'})^2 (\frac{d}{d'}) | \Disc(x')}} 1.
\end{equation}
By Lemma \ref{lemma_BBP2}, this is
\begin{equation}
\ll
\frac{q^{\epsilon} c' d'}{a^2 b^3 c^3 d^4} \bigg(
\frac{Y}{a^8 b^4 c'^2 c^4 d'^3 d^3} + \bigg( \frac{cd}{c' d'} \bigg)^2 \bigg)
\ll q^{\epsilon} \frac{Y}{a^{10} b^7 c^7 c' d^7 d'^2}  + \frac{q^{\epsilon}}{a^2 b^3 c c' d^2 d'}
\ll q^{-7 + \epsilon} Y + q^{-1 + \epsilon}.
\end{equation}
The theorem now follows by summing over the $6^{\omega(q)} \ll q^{\epsilon}$ possibilities for $a, b, c, c', d, d'$.
\end{proof}

\begin{proof}[Proof of Lemma \ref{lemma_BBP2}]
The proof follows the methods presented in this paper, but it is simpler. As the proof is very similar, we will
omit some of the details.
 
We begin by defining the {\itshape $d$-divisible Shintani zeta functions}, which
count only those discriminants divisible by $d = rs^2$. For prime factors of $s$ (other than 2, 3)
the local condition is given by the complement of $V_p$, and for factors of prime $r$ it is described in \cite{TT}.
These zeta functions satisfy a close analogue of Theorem \ref{thm_DWT}; we omit the details, but the exact form 
of the functional equation can be readily deduced from \cite{TT}.

For cubefree $d = rs^2$,
let  $\xi^{\pm}_d(s) := \sum_{d | n} a^{\pm}(n) n^{-s}$ denote\footnote{This is the same notation that we used for the
$q$-nonmaximal zeta function. This notation will be used only in the proof of this lemma.}
the $d$-divisible Shintani zeta function, and let
$\xi_d(s) := \sum_{d | n} a(n) n^{-s}$ denote either of the diagonalized zeta functions, as in \eqref{eqn_new_diag}.
Then we have the relation
\begin{equation}\label{eqn_smoothed_integral}
\sum_{d | n} a(n) \exp(- n/X) = \int_{c - i \infty}^{c + i \infty}
\xi_d(s) X^s \Gamma(s) ds
\end{equation}
for any $c \in (1, \frac{3}{2})$ (analogously to \eqref{eqn_for_each_q}). 
Note that $\sum_{d | n; \ n < X} a(n) \ll \sum_{d | n} a(n) \exp(-n/X)$; 
the smoothing factor of $\exp(-n/X)$ is introduced to improve
the error terms.

As $\xi_d(s) \Gamma(s)$ is holomorphic for $\Re s > -1/2$, except for poles at $s = 1, \frac{5}{6}, 0$,
we may shift the contour and use the functional equation we obtain 
that
\begin{multline}
\sum_{d | n} a(n) \exp(- n/X) =
X^{1 - c} \sum_{\mu_n} \frac{b_d(\mu_n)}{\mu_n^c} \int_{c - i \infty}^{c + i \infty}
\frac{\Delta(s)}{\Delta(1 - s)} (X \mu_n)^{-it} \Gamma(1 - s) ds
+
\\
\bigg( \Gamma(1) \Res_{s = 1} \xi_d(s) \bigg) X + 
\bigg( \Gamma(5/6) \Res_{s = 5/6} \xi_d(s) \bigg) X^{5/6} + 
\xi_d(0)
\end{multline}
where $\sum_{\mu_n \in \frac{1}{d^4} \Z} b_d(\mu_n) \mu_n^{-s}$ is the dual zeta function, $\Delta(s)$ is as in \eqref{eqn_def_delta},
the residues are $\ll d^{-1 + \epsilon}$ and $d^{-5/6 + \epsilon}$ respectively, and $\xi_d(0) \ll d$,
as follows from \cite{TT}.
The integral above is absolutely convergent, and we bound it by an absolute constant (which in particular does not depend
on $\mu_n$). We must therefore bound
$X^{1- c} \sum_{\mu_n} {b_d(\mu_n)} \mu_n^{-c}.$

We choose $c = 1 + \frac{\epsilon}{4}$ so that $X^{1 - c} \leq (d^{-4})^{-\epsilon/4} = d^{\epsilon}$.
Arguing as in Proposition \ref{prop_high_middle_bound}, we see that the sum over $\mu_n$ will be
$\ll (rs)^{2 + \epsilon}$, implying the lemma, provided we can show that 
\begin{equation}\label{eqn_circular_bound}
\sum_{\mu_n < Y} |b_d(\mu_n)| \ll (rs)^{2 + \epsilon} Y.
\end{equation}
Note that when $d = s^2$ and $(d, 6) = 1$, this is {\itshape exactly} \eqref{eqn_trivial_bound_2}.
Crucially, the proof of \eqref{eqn_trivial_bound_2} does not depend on this lemma, but 
our argument does have an interesting circular flavor: the bound \eqref{eqn_trivial_bound_2} is an essential ingredient
in the proof of the stronger bound \eqref{eqn_trivial_bound_3}.
As we discussed earlier,
the idea is that an important piece of the Fourier transform $\widehat{\Phi}_q(x)$ resembles $\Phi_q(x)$ itself.

To prove \eqref{eqn_circular_bound}, note that as before, it suffices to prove that
\begin{equation}\label{eqn_circular_bound2}
\sum_{|\Disc(x)| < Y} |\widehat{\Phi}_d(x)| \ll d^{-4} (rs)^{2 + \epsilon} Y,
\end{equation}
where $\widehat{\Phi}_d(x)$ is again multiplicative in $d = r s^2$. We extend the proof of \eqref{eqn_trivial_bound_2}
(which extended the proof of Theorem \ref{thm_trivial_bound}) to cover the case $r > 1$. For each factorization $r = ef$,
consider the contribution
from those $x$ for which $(\ctt(x), r) = e$. We divide each such $x$ by $e$,
and the formulas in \cite{TT} imply that $|\widehat{\Phi}_f(x)| \leq f^{-2}$, so the total $r$-contribution to \eqref{eqn_circular_bound2}
is  $\leq e^{-4} f^{-2} \leq r^{-2}$ for each factorization $r = ef$. Summing over the $\ll r^{\epsilon}$ such factorizations, we obtain
a total $r$-contribution $\ll r^{-2 + \epsilon}$, as claimed in
\eqref{eqn_circular_bound2}. This completes the proof.
\end{proof}

\section{The proof of Roberts' conjecture}\label{sec_proof}
We will prove Roberts' conjecture in three steps. In Section \ref{subsec_reduction} we discuss the relationship
between the Shintani zeta coefficients and counting functions for cubic rings, and reduce Roberts' conjecture
to a statement about partial sums of Shintani zeta functions. In Section \ref{subsec_setup} we incorporate the
Datskovsky-Wright diagonalization, and transform our problem into one that can be readily addressed using
a contour integration argument of Chandrasekharan and Narasimhan \cite{CN}. Finally, in Section \ref{subsec_contour} we
do this contour integration. As it would be impractical to reproduce the entire argument in \cite{CN}, we will refer to their paper for many of the details and call the reader's attention to the few changes we introduce to their argument.

\subsection{Reduction to Shintani zeta coefficients.}\label{subsec_reduction}

We want to obtain estimates for $N_3^{\pm}(X)$, the count of cubic fields of positive or negative discriminant
less than $X$. The first step in our argument is to 
relate these quantities to partial sums of the coefficients of the Shintani zeta
function. Define Dirichlet series $F^{\pm}(s) = \sum_n c^{\pm}(n) n^{-s}$ by
\begin{equation}\label{eqn_field_shintani}
F^{\pm}(s) = \sum_{n \geq 1} c^{\pm}(n) n^{-s} := \sump_{x \in \SL_2(\Z) \backslash V_{\Z}} \frac{1}{|\Stab(x)|} |\Disc(x)|^{-s},
\end{equation}
where the prime indicates that the sum is restricted to those $x$ which are maximal at all places (i.e., contained in 
$U_p$ for all $p$).

We define partial sums
\begin{equation}
N^{\pm}(X) := \sum_{n \leq X} c^{\pm}(n).
\end{equation}

We will prove the following:

\begin{proposition}\label{prop_shintani_reduction} We have
\begin{equation}
N_3^{\pm}(X) = \frac{1}{2} N^{\pm}(X) - \frac{3}{\pi^2} X + O(X^{1/2}).
\end{equation}
\end{proposition}

\begin{proof}
By the Delone-Faddeev correspondence (see also Section 2 of \cite{DW2}), 
the Dirichlet series in \eqref{eqn_field_shintani} counts fields of degree
$\leq 3$ (or, more properly, their maximal orders), with different weights for different types of fields. Non-Galois
cubic fields are counted with weight $2$, Galois fields are counted with weight $2/3$, quadratic fields are counted
with weight 1, and $\Q$ is counted with weight $1/3$.

The number of cyclic cubic extensions of discriminant $\leq X$ is $O(X^{1/2})$ \cite{C}, the number of quadratic extensions
of either positive or negative discriminant $\leq X$ is equal to $\frac{3}{\pi^2} X + O(X^{1/2})$,
and of course there is only one trivial extension of $\Q$. The result therefore follows by subtracting and reweighting 
these contributions as appropriate.
\end{proof}

\subsection{Setup for the contour integration}\label{subsec_setup}
In this section we will incorporate the inclusion-exclusion sieve and Datskovsky and Wright's diagonalization, and
bring our problem to a form where we can apply contour integration.

By Proposition \ref{prop_shintani_reduction}, it suffices to count
\begin{equation}\sump_{\substack{x \in \SL_2(\Z) \backslash V_{\Z} \\ \pm \Disc(x) \leq X}}
\frac{1}{\Stab(x)},
\end{equation}
where the dash on the sum indicates that we count only those lattice points corresponding to {\itshape maximal}
cubic rings. 
A cubic ring is maximal if and only if it is maximal at each prime. By inclusion-exclusion, this sum is equal to 
\begin{equation}
\sum_q \mu(q) \bigg( \sum_{n \leq X} a^{\pm}_q(n) \bigg),
\end{equation}
where the $a^{\pm}_q(n)$ are the coefficients of the $q$-nonmaximal Shintani zeta functions. By Lemma \ref{lemma_BBP},
the inner sum is $\ll X q^{-2 + \epsilon}$, uniformly in $q$, and it follows that the total sum is
\begin{equation}
\sum_{q \leq Q} \mu(q) \bigg( \sum_{n \leq X} a^{\pm}_q(n) \bigg) + O\bigg(\frac{X}{Q^{1 - \epsilon}}\bigg),
\end{equation}
for any choice of $Q$. The main term above is what we want to estimate.

Although it is not strictly necessary (see Theorem 3 of \cite{SS}),
it will simplify our computations to incorporate Datskovsky and Wright's diagonalization, described in Section \ref{subsec_diag}.
We will write
\begin{equation}\label{eqn_new_diag}
a_q(n) := \sqrt{3} a^+_q(n) \pm a^-_q(n),
\end{equation}
such that the zeta functions $\xi_q(s) := \sum_n a_q(n) n^{-s}$ satisfy
the simple functional equation \eqref{eqn_diag_fe} or \eqref{eqn_diag_fe2}. As we will prove our results simultaneously
for both 
choices of sign in \eqref{eqn_new_diag}, we will not indicate this sign in our notation.

We write $N(X)$ for either of the analogous linear combinations of $N^{\pm}(X)$, and we will prove estimates for 
\begin{equation}\label{eqn_sum_ie}
N_Q(X) := \sum_{q \leq Q} \mu(q) \bigg( \sum_{n \leq X} a_q(n) \bigg).
\end{equation}
We then take the appropriate linear combinations to recover the analogous
estimates for the original Shintani zeta function.

To evaluate \eqref{eqn_sum_ie}, recall that Perron's formula yields the equality\footnote{For strict equality, we must take
$X$ not equal to any value of $n$ (any irrational number will do).}
\begin{equation}\label{eqn_perron}
\sum_{n \leq X} a_q(n) = \int_{c - i \infty}^{c + i \infty} \xi_q(s) \frac{X^s}{s} ds
\end{equation}
for any $c > 1$. {\itshape In principle}, one evaluates the integral by shifting the contour to the left, obtaining 
main terms of order $X$ and $X^{5/6}$ from the poles of $\xi_q(s)$, along with an error term. 
In practice, one runs into convergence issues at infinity
and must tweak the method somehow. We adopt the method of Chandrasekharan and Narasimhan \cite{CN}, 
which has its origins in work of Landau \cite{L}. In particular, following \cite{CN}, we will 
smooth the sum above to obtain an integral with nice convergence
properties at infinity, and then use a finite differencing method to recover the sum in \eqref{eqn_perron} from the smoothed sum.

As we will see, we may improve our error terms by departing from \cite{CN} in one respect. 
We will 
smooth the entire sum in \eqref{eqn_sum_ie}, 
estimate the smoothed sum over each $q$ separately, and 
combine the contributions from all $q$ to obtain a smoothed version of the count in \eqref{eqn_sum_ie}.
Recovering the count in \eqref{eqn_sum_ie} from the smoothed count involves an error term, and
the error made in unsmoothing the combined count is roughly equal to the error made in unsmoothing
the contribution from any individual $q$. Therefore, we will not actually estimate the contribution of any
individual $q$ to \eqref{eqn_sum_ie}.

\subsection{The contour integration}\label{subsec_contour}

We now begin in earnest, closely following \cite{CN}. We introduce a smoothing factor $(X - n)^{\rho}$, and write
\begin{equation}
N_Q^{\rho}(X) := \frac{1}{\Gamma(\rho + 1)} \sum_{q \leq Q} \mu(q) \bigg( \sum_{n \leq X} (X - n)^{\rho} a_q(n) \bigg).
\end{equation}
Here $\rho$ is any sufficiently large integer. We may in fact take $\rho = 3$, but to follow the notation of \cite{CN}\footnote{In the notation of \cite{CN}, we have $A = 2, q =1, r =1, \delta = 1,$
and $N = 4$, as determined by the structure of our problem.}
we will leave the value undetermined.
 (Any error terms may depend on $\rho$.)

For each $q$, we have
\begin{equation}\label{eqn_for_each_q}
\frac{1}{\Gamma(\rho + 1)} \sum_{n \leq X} (X - n)^{\rho} a_q(n) =
\frac{1}{2 \pi i} \int_{c - i \infty}^{c + i \infty}
\frac{1}{s (s + 1) \cdots (s + \rho)} \xi_q(s) X^{s + \rho} ds,
\end{equation}
for any $c > 1$. We move the integral to the line $\sigma = 1 - c$, choosing $c < \frac{5}{4}$ so that we do not
pick up any singularities of the integral left of $s = 0$, and so that the integral \eqref{eqn_fe_integral} converges for $\rho \geq 3$.
In doing so, we
pick up contributions from the residues of $\xi_q(s)$ at $s = 1$ and $s = 5/6$.

Later, we will estimate the integral on the line $\sigma = 1 - c$ using the functional equation. We first explain
how $N_Q^{\rho}(X)$ is related to our unsmoothed count $N_Q(X)$. For  
a parameter $y$ to be determined later, define a finite differencing operator $\Delta_y^{\rho}$ (on the space of
real valued functions $F$) by
$$\Delta_y^{\rho}F(x) := \sum_{\nu = 0}^{\rho} (-1)^{\rho - \nu} {\rho \choose \nu} F(x + \nu y).$$
It is proved in (4.14) of \cite{CN}
that
\begin{equation}\label{eqn_cn414}
\Delta_y^{\rho} [N_Q^{\rho}(X) - R_Q^{\rho}(X)]
= y^{\rho} [N_Q(X) - R_Q(X)] + O \bigg(y^{\rho + 1} + y^{\rho} 
\sum_{X < n \leq X + \rho y} \sum_{q \leq Q} a_q(n) \bigg),
\end{equation}
where
$$R_Q^{\rho}(x) = \sum_{q \leq Q} \mu(q)
\bigg( \frac{1}{1 (1 + 1) \cdots (1 + \rho)} X^{1 + \rho} \Res_{s = 1} \xi_q(s)
+ \frac{1}{\frac{5}{6} \Big( \frac{5}{6} + 1 \Big) \cdots \Big( \frac{5}{6} + \rho \Big)}
X^{5/6 + \rho} \Res_{s = 5/6} \xi_q(s) \bigg)$$
(with an additional residue term at $s = 0$ which we subsume into our error term), and
$$R_Q(X) = \sum_{q \leq Q} \mu(q)
\bigg( X \Res_{s = 1} \xi_q(s) + \frac{6}{5} X^{5/6} \Res_{s = 5/6} \xi_q(s) \bigg).$$

The error term in \eqref{eqn_cn414} is $O(y^{\rho + 1 + \epsilon})$ if $y > X^{3/5}$;
this follows by estimating
$$\sum_{X < n \leq X + \rho y} \sum_{q \leq Q} a_q(n)
\ll y^{\epsilon} \sum_{X < n \leq X + \rho y} a(n) \ll y^{1 + \epsilon}.$$
The first estimate follows because $|a_q(n)| \leq |a(n)|$ and $a_q(n) = 0$ unless $q^2 | n$,
and the latter estimate follows from partial sum estimates for the standard Shintani zeta function.

Therefore, for $y > X^{3/5}$ it follows that 
\begin{equation}\label{eqn_CN_error}
N(X) - R_Q(X)
\ll y^{- \rho} \Delta_y^{\rho} [N_Q^{\rho}(X) - R_Q^{\rho}(X)]
+ y^{1 + \epsilon} + \frac{X}{Q^{1 - \epsilon}},
\end{equation}
where
\begin{equation}
N_Q^{\rho}(X) - R_Q^{\rho}(X) = \sum_{q \leq Q}
\bigg(
\frac{1}{2 \pi i} \int_{1 - c - i \infty}^{1 - c + i \infty}
\frac{1}{ s (s + 1) \cdots (s + \rho)} \xi_q(s) X^{s + \rho} ds \bigg).
\end{equation}
We will study this integral individually for each $q$. 
To denote this, we replace the subscript $Q$ with $q$ throughout.
Applying the functional equation \eqref{eqn_diag_fe} or \eqref{eqn_diag_fe2}, the integral is equal to
\begin{equation}\label{eqn_fe_integral}
\frac{1}{2 \pi i} 
\int_{c - i \infty}^{c + i \infty}
\frac{1}{(1 - s) (2 - s ) \cdots (1 + \rho - s)} \frac{\pm \Delta(s)}{3 \Delta(1 - s)} 
\widehat{\xi}_q(s) X^{1 + \rho - s} ds,
\end{equation}
where
\begin{equation}\label{eqn_def_delta}
\Delta(s) := \bigg(\frac{2^4 \cdot 3^6}{\pi^4} \bigg)^{s/2}
 \Gamma\bigg(\frac{s}{2}\bigg)
 \Gamma\bigg(\frac{s}{2} + \frac{1}{2} \bigg) \Gamma\bigg(\frac{s}{2} + a_3 \bigg) \Gamma\bigg(\frac{s}{2} + a_4\bigg).
\end{equation}
Here $a_3$ and $a_4$ are equal to either $5/12$ and $7/12$ or $\pm 1/12$ as appropriate.

The integral in \eqref{eqn_fe_integral} is equal to
\begin{equation}\label{eqn_first_bq}
\sum_{\mu_n \in \frac{1}{q^8} \Z} \frac{b_q(\mu_n)}{{\mu_n}^{1 + \rho}}
\Bigg(
\frac{1}{2 \pi i}
\int_{c - i \infty}^{c + i \infty}
\frac{1}{(1 - s) (2 - s ) \cdots (1 + \rho - s)} \frac{\pm \Delta(s)}{3 \Delta(1 - s)} 
(\mu_n X)^{1 + \rho - s} ds
\Bigg)
.
\end{equation}
This integral and its finite difference are thoroughly analyzed in \cite{CN}. Although one
might hope to play the oscillation of the $b_q(\mu_n)$ against oscillation in this integral,
our attempts to do this were unsuccessful. However, we still obtain good error terms
by taking absolute values of the $b_q(\mu_n)$ and using bounds for the integral proved in \cite{CN}.

Recall that our error term in \eqref{eqn_CN_error} consists of a sum over $q$ of the operator $\Delta_y^{\rho}$
applied to this integral. Following the argument in \cite{CN}, and in particular the bounds on p. 109 there,
we have
\begin{equation}\label{eqn_from_CN}
\Delta_y^{\rho} [N_q^{\rho}(X) - R_q^{\rho}(X)]
\ll 
y^{\rho} X^{3/8} \sum_{\mu_n \leq z} \frac{|b_q(\mu_n)|}{\mu_n^{5/8}}
+ X^{3/8 + 3 \rho/4} \sum_{\mu_n > z} \frac{|b_q(\mu_n)|}{\mu_n^{5/8 + \rho/4}},
\end{equation}
where $z$ is a free parameter. 
We estimate the sums on the right
using the bounds given in 
Proposition \ref{prop_high_middle_bound}.
We conclude that
\begin{equation}\label{eqn_used_CN}
y^{- \rho} \Delta_y^{\rho} [N_q^{\rho}(X) - R_q^{\rho}(X)]
\ll q^{1 + \epsilon} X^{3/8} z^{3/8} \bigg( 1 + \bigg(\frac{X^3}{y^4 z}\bigg)^{\rho/4} \bigg),
\end{equation}
and therefore, adding over all $q$,
\begin{equation}\label{eqn_n_estimate}
N(X) - R_Q(X) \ll Q^{2 + \epsilon}
X^{3/8} z^{3/8} \bigg(1 + \bigg(\frac{X^3}{y^4 z}\bigg)^{\rho/4} \bigg)
+ y^{1 + \epsilon}
+ \frac{X}{Q^{1 - \epsilon}}.
\end{equation}
We choose $y = X/Q$ and $z = X^3 / y^4$ to equalize error terms. The above is then
\begin{equation}\label{eqn_choose_Q}
\ll Q^{7/2 + \epsilon} + X^{1 + \epsilon}/Q,
\end{equation}
and choose $Q = X^{2/9}$ to obtain an error term of $O(X^{7/9 + \epsilon}).$ Note that $y > X^{3/5}$ as
required for \eqref{eqn_CN_error}.

We now reverse our diagonalizations to obtain estimates for $N^{\pm}(X)$, with the same error terms. 
It remains to evaluate $R_Q^{\pm}(X)$.
 We see that
\begin{equation}
R_Q^{\pm}(X) = X \sum_{q \leq Q} \mu(q) \Res_{s = 1} \xi^{\pm}_q(s)
+ \frac{6}{5} X^{5/6} \sum_{q \leq Q} \mu(q) \Res_{s = 5/6} \xi^{\pm}_q(s).
\end{equation}

We now apply the formulas in \cite{TT} for the residues, quoted in Theorem \ref{thm_DWT}. 
We have
\begin{multline}\label{eqn_plugged_in}
R_Q^{\pm}(X) = X \sum_{q \leq Q}
\mu(q) \bigg( \alpha^{\pm} \prod_{p | q} \bigg(\frac{1}{p^2} + \frac{1}{p^3} - 
\frac{1}{p^5} \bigg) + \beta \prod_{p | q} \bigg(\frac{2}{p^2}
- \frac{1}{p^4} \bigg) \bigg)
\\
+ \frac{6}{5} \gamma^{\pm} X^{5/6} \bigg( \sum_{q \leq Q} \mu(q) \prod_{p | q}
\bigg(\frac{1}{p^{5/3}} + \frac{1}{p^2} - \frac{1}{p^{11/3}}\bigg) \bigg),
\end{multline}
where $\alpha^+ = \pi^2/36$, $\alpha^- = \pi^2 / 12$, $\beta = \pi^2/12$,
$\gamma^+ = \frac{\Gamma(1/3)^3 \zeta(1/3)}{4 \sqrt{3} \pi}$, and $\gamma^- = \sqrt{3} \gamma^+$.

We replace the sums over $q \leq Q$ by the appropriate Euler products, with error 
$\ll X Q^{-1 + \epsilon} + X^{5/6} Q^{-2/3 + \epsilon} \ll X^{7/9 + \epsilon}$,
and we see that
\begin{equation}\label{eqn_rq_estimate}
R_Q^{\pm}(X) = X \bigg( \alpha^{\pm} \frac{1}{\zeta(2) \zeta(3)}
+ \beta \frac{1}{\zeta(2)^2} \bigg)
+ X^{5/6} \bigg( \frac{6}{5} \gamma^{\pm} \cdot \frac{1}{\zeta(2) \zeta(5/3)} \bigg)
+ O(X^{7/9 + \epsilon}).
\end{equation}
Theorem \ref{thm_rc} now follows by combining \eqref{eqn_n_estimate} and \eqref{eqn_rq_estimate}
with Proposition \ref{prop_shintani_reduction}.

\section{Generalizations of Roberts' conjecture}\label{sec_generalizations}

The proofs of our generalizations of Roberts' conjecture follow along very similar lines. In this section we will
describe these generalizations more explicitly, and explain the new steps required in the proofs.

As we prove a variety of generalizations and work out some explicit examples, this section is rather long. 
We begin in Section \ref{subsec_torsion} with the proof of Theorem \ref{thm_rc_torsion}, on 3-torsion in quadratic
fields. In Section \ref{subsec_local_specs}
we describe the proof of Theorem \ref{thm_rc_extended} concerning local specifications, and in Section
\ref{subsec_local_torsion} we extend these arguments to the 3-torsion problem (Theorem \ref{thm_rc_torsion2}).

In Section \ref{subsec_ap} we tackle the problem of arithmetic progressions, and we prove Theorem
\ref{thm_technical_ap}, our most general result on cubic fields. In Section \ref{subsec_computations}
we use Theorem \ref{thm_technical_ap} 
work out  some examples in more detail, and in particular we prove Theorem \ref{thm_progressions}.
We also present the results of some numerical calculations. We conclude in Section \ref{subsec_torsion_ap}
with a general result on 3-torsion in arithmetic progressions (Theorem \ref{thm_technical_torsion_ap}) and the 
proof of Theorem \ref{thm_torsion_progressions}.

\subsection{3-torsion in quadratic fields}\label{subsec_torsion}

As in \cite{DH} and \cite{BBP}, we use the folllowing classical result of Hasse, which is proved using class field theory:
There is a bijection between pairs of {\itshape nontrivial} 3-torsion elements in quadratic fields $L$ with
$0 < \pm \Disc(L) < X$, and cubic fields $K$ with $0 < \pm \Disc(K) < X$ which are not totally ramified at any prime.
Under this bijection $\Disc(L) = \Disc(K)$, and the condition on $K$ is equivalent to requiring that $\Disc(K)$
be fundamental.

It therefore suffices to count cubic fields which are nowhere totally ramified. Write $M_3^{\pm}(X)$ and $M_3^{\pm}(q, X)$
for the counting functions of such fields. We may count these fields by replacing the $q$-nonmaximal zeta functions
with ``$q$-nonmaximal-or-totally-ramified'' zeta functions; in the langauge of Davenport and Heilbronn, we shrink
the set $U_p$, defined in Proposition \ref{def_up}, to a new set $V_p$, which excludes cubic orders which are
totally ramified at $p$. We then adjust Definition \ref{def_q_nonmax} to incorporate this condition, and we write
$a'_q(n)$ for the coefficients of our modified zeta function. 

We may estimate $M_3^{\pm}(X)$ using the same proof. By Proposition \ref{prop_shintani_reduction}, we have
\begin{equation}\label{eqn_torsion_reformulation}
M_3^{\pm}(X) = \frac{1}{2} \sum_{q \geq 1} \mu(q) \bigg( \sum_{n \leq x} a'_q(n) \bigg) - \frac{3}{\pi^2} X + O(X^{1/2}),
\end{equation}
and Lemma \ref{lemma_BBP} establishes that we may again truncate the sum to $q \leq Q$ with error $\ll X/Q^{1 - \epsilon}$.

We estimate the sums of the $a'_q(n)$ in the same way as before. The cubic Gauss sum appearing implicitly in the analogue of
\eqref{eqn_first_bq} is a little bit different, and we apply the bounds in Proposition \ref{prop_high_middle_bound2}.
In place of \eqref{eqn_used_CN}, we obtain

\begin{equation}\label{eqn_used_CN2}
y^{- \rho} \Delta_y^{\rho} [N_q^{\rho}(X) - R_q^{\rho}(X)]
\ll q^{7/8 + \epsilon} X^{3/8} + q^{1 + \epsilon} X^{3/8} z^{3/8} \bigg( 1 + \bigg(\frac{X^3}{y^4 z}\bigg)^{\rho/4} \bigg),
\end{equation}
as long as $z \geq q^{-2}$.

We split the sum over $q \leq Q$ into, say, $q \leq X^{1/6}$ and $X^{1/6} < q \leq Q$. We choose $y = X/Q$ as before.
In the range $q \leq X^{1/6}$ we choose $z = 1$, and the contribution of this range to $M(X) - R_Q(X)$ is
\begin{equation}
\ll (X^{1/6})^{15/8 + \epsilon} X^{3/8} + 
(X^{1/6})^{2 + \epsilon} X^{3/8}  \bigg( 1 + \bigg(\frac{X^3}{y^4}\bigg)^{\rho/4} \bigg),
\end{equation}
which is bounded by $X^{17/24 + \epsilon}$ so long as $y \geq X^{3/4}$. 
In the range  $X^{1/6} < q \leq Q$, we choose
$z = X^3 / y^4$ as before, and check that $z \geq q^{-2}$ for each $q$. We obtain a contribution to 
$M(X) - R_Q(X)$ of
\begin{equation}\label{eqn_chooseQ_1a}
Q^{15/8 + \epsilon} X^{3/8} + Q^{7/2 + \epsilon} + X^{1 + \epsilon}/Q.
\end{equation}
Because of the new first term, the optimal choice is $Q = X^{5/23}$, which gives an error
term of $X^{18/23 + \epsilon}$.

The remainder of the machinery
of Section \ref{sec_proof} works unchanged. We compute the residues of the new Shintani zeta functions using the tables
in \cite{TT}. We obtain, analogously to \eqref{eqn_plugged_in},

\begin{multline}\label{eqn_plugged_in2}
R_Q^{\pm}(X) = X \sum_{q \leq Q}
\mu(q) \bigg( \alpha^{\pm} \prod_{p | q} \bigg(\frac{2}{p^2} - \frac{1}{p^4} 
\bigg) + \beta \prod_{p | q} \bigg(\frac{2}{p^2}
- \frac{1}{p^4} \bigg) \bigg)
\\
+ \frac{6}{5} \gamma^{\pm} X^{5/6} \bigg( \sum_{q \leq Q} \mu(q) \prod_{p | q}
\bigg(\frac{1}{p^{5/3}} + \frac{2}{p^2} - \frac{1}{p^{8/3}} - \frac{1}{p^3}\bigg) \bigg),
\end{multline}
and therefore
\begin{equation}\label{eqn_torsion_alternate}
M_3^{\pm}(X) = \frac{1}{2 \zeta(2)^2} \alpha^{\pm} X +
\frac{3}{5} \gamma^{\pm} X^{5/6} \prod_p \bigg(
1 - \frac{1}{p^{5/3}} - \frac{2}{p^2} + \frac{1}{p^{8/3}}+ \frac{1}{p^3} \bigg)
+ O(X^{18/23 + \epsilon}).
\end{equation}
This is equivalent to \eqref{eqn_torsion1} by the formulas $\zeta(2) = \frac{\pi^2}{6}$ and $\Gamma(1/3)\Gamma(2/3) =
\frac{2 \pi}{\sqrt{3}}$.

\begin{remark} As we remarked previously, we can obtain an error term of $O(X^{9/11 + \epsilon})$
without appealing to the more difficult results of Section \ref{section_torsion_bounds}. To do this, we use only the simple bound
\eqref{eqn_trivial_bound_2} from Section \ref{section_torsion_bounds}; equivalently, we observe that when we 
replace $U_p$ with $V_p$, this multiplies each term in \eqref{eqn_from_CN} by a factor of at most $q$. (This follows
from observing that the bounds in Lemma \ref{lem_gauss_sum2} are at most $p$ times those in Lemma \ref{lem_gauss_sum}.)
This yields, in place of 
\eqref{eqn_choose_Q} and \eqref{eqn_chooseQ_1a}, the bound
\begin{equation}
M(X) - R_Q(X) \ll Q^{9/2 + \epsilon} + X^{1 + \epsilon}/Q,
\end{equation}
and we obtain an error term of $X^{9/11 + \epsilon}$ by choosing $Q = X^{2/11}$.
\end{remark}

\subsection{Generalizations involving local conditions}\label{subsec_local_specs} In this section we will discuss the proof of Theorem \ref{thm_rc_extended}.
By a
{\itshape local specification} $\calS_p$ at $p$ we mean a choice of one or more maximal cubic rings $R/\Z_p$, and
we say that a cubic field $K$ satisfies $\calS_p$ if
$\calO_K \otimes \Z_p$ is isomorphic to one of these $R$. 
For each $p$, there are finitely many possibilities for $R$, and they
may be detected by the Delone-Faddeev correspondence modulo 16 (if $p = 2$), 27 ($p = 3$), or $p^2$ $(p > 3)$.

\begin{remark}
Our methods also allow us to count nonmaximal cubic orders with various conditions, but for the sake of
simplicity we have excluded this possibility.
\end{remark}

The possibilities for $R$ are in bijection with extensions of $\Q_p$ of degree at most 3, 
which have been completely classified. For the classification we refer to the comprehensive paper and database of Jones
and Roberts \cite{JR}. We also note that the framework we describe here appeared in Roberts' paper \cite{R}.

In the tables that follow we list the following information: We list all possibilities for $R$, and we recall that 
different choices of $R$ 
detect the different splitting types of $p$ in $K$. If $p$ is totally
split, partially split, or inert, then $R$ is respectively equal to 
$\Z_p^3$, $\Z_p \times \calO_F$, or $\calO_L$, where $\calO_F$ and $\calO_L$ are the integer rings of the unique 
unramified quadratic and cubic extensions of $\Q_p$.
If $p$ is partially or totally ramified, then
$R$ is
$\Z_p \times \calO_{F'}$ or $\calO_{L'}$, where $\calO_{F'}$ and $\calO_{L'}$ are the integer rings of
ramified quadratic and cubic extensions of $\Q_p$. 
Depending on the value of $p$, there may be multiple possibilities for $F'$ and $L'$, and we list polynomials
generating each possible extension.

We also list the ``conductor'' $p^e$ for each choice of $R$. By this we mean the following: Suppose that $x \in V_{\Z}$ corresponds to a cubic
ring $\calO / \Z$. Then we say that $R$ has conductor $p^e$ if the condition $\calO \otimes \Z_p \cong R$ may be
detected by reducing $x$ modulo $p^e$, and if $p^e$ is the minimal integer with this property. We have $e \leq 4$ in all cases and $e \leq 2$ if $p > 3$, 
and these quantities naturally appear in our error terms.

Finally, we list the {\itshape local densities} at $s = 1$ and $s = 5/6$. The densities in the table are
unnormalized, and for each prime $p$ we normalize by dividing by the normalizing factors
\begin{equation}\label{def_cpkp1}
C_p := 1 + \frac{1}{p} + \frac{1}{p^2}, \ \ \
K_p := \frac{(1 - p^{-5/3})(1 + p^{-1})}{1 - p^{-1/3}}
\end{equation}
for $s =1$ and $s = 5/6$. (These quantities are simply the sum of the local densities.)
The constants $C(\calS)$ and $K(\calS)$ appearing in Theorem \ref{thm_rc_extended} are then 
given by the products of the normalized local densities at $s = 1$ and $s = 5/6$ respectively.
We also include a factor
of $C^{\pm}$ or $K^{\pm}$ according to the sign of the discriminants being counted; in light of the adelic
origin of our residue formulas, one should consider this choice of sign to be a local specification at the infinite
place.

\begin{remark} The normalized local density at $s = 1$ has a simple geometric interpretation. Recall
that our count of cubic fields incorporated, for each prime $p$, a factor of $(1 - p^{-2})(1 - p^{-3})$
corresponding to the proportion of cubic rings which are maximal at $p$. The density at $p$ is
simply the proportion of maximal cubic rings which have a given splitting type. Under the Delone-Faddeev
correspondence, this may then be
determined by counting $\GL_2(\Z / p^e \Z)$ orbits on $V_{\Z / p^e \Z}$.

To give an example, we compute the density of rings which are totally split at a prime $p$.
Under Delone-Faddeev, these consist of a single $\GL_2(\Z / p \Z)$-orbit on $V_{\Z / p \Z}$ whose stabilizer
has order 6. We then
verify that
\begin{equation}\label{example_density}
\frac{\frac{1}{6} \# \GL_2(\Z / p \Z)}{p^4}
= \frac{1}{6} \bigg(1 - \frac{1}{p} \bigg) \bigg(1 - \frac{1}{p^2} \bigg)
= \frac{1/6}{1 + p^{-1} + p^{-2}}
 \cdot (1 - p^{-2}) (1 - p^{-3} ).
\end{equation}
The {\itshape unnormalized} density of any cubic ring $R/\Z_p$ is equal to $\frac{1}{|\Disc(R)|_p |\Aut(R)|},$
and the normalization factor has a geometric interpretation which is described in \cite{DW2} or Proposition 8.8 of \cite{TT}.
The density at $s = 5/6$ can also be interpreted in a similar but more complicated way.
\end{remark}
This brings us now to our tables:
\begin{center}
\begin{tabular}{ l | c c  c }
Condition at $p$ & \textnormal{Conductor} & \textnormal{Density at } $s = 1$ & \textnormal{Density at} $s = 5/6$ \\ \hline
\textnormal{Totally split} & $p$ &
$1/6$ & $(1 + p^{-1/3})^3/6$  \\ \hline
\textnormal{Partially split} & $p$ &
$1/2$ & $(1 + p^{-1/3})(1 + p^{-2/3})/2$ \\ \hline
\textnormal{Inert} & $p$ &
$1/3$ & $(1 + p^{-1})/3$ \\ \hline
\textnormal{Partially ramified} & $p^2$ &
$1/p$ & $(1 + p^{-1/3})^2 / p$ \\
$\ \ (p \neq 2) \ x^2 + a u^2 p$ & $p^2$ & $\times \frac{1}{2}$ & $\times \frac{1}{2}$ \\ \hline
\textnormal{Totally ramified} & $p^2$ &
$1/p^2$ & $(1+ p^{-1/3}) / p^2$ \\
$\ \ (p \equiv 2 \ (\textmod \ 3))$ & $p^2$ & $\times 1$ & $\times 1$ \\
$\ \ (p \equiv 1 \ (\textmod \ 3)) \ x^3 + a u^3 p$ & $p^2$ & $\times \frac{1}{3}$ & $\times \frac{1}{3}$
\end{tabular}
\end{center}
In the ramified case, the fields generated by $x^2 + au^2$ and $x^3 + au^3$ are isomorphic for
any $u \in (\Z / p \Z)^{\times}$, but are distinct as $a$ ranges over the quadratic or cubic residue classes.
The notation $\times \frac{1}{2}$ (for example) means that the local density is halved for each of the two
cases.

\begin{remark} The densities at $s = 5/6$ appear in a modified form in Proposition 5.3 of Datskovsky-Wright \cite{DW2}.
All of our subsequent density tables also depend closely on Datskovsky and Wright's work.
\end{remark}

At $p =2$ and $p = 3$ there are additional possibilities, because there 
are more ramified maximal quadratic rings over $\Z_2$ and cubic rings over $\Z_3$. We list all of
the possibilities in the following tables, following the database \cite{JR}. We list each ring by giving
a generating polynomial over $\Z_2$ or $\Z_3$. (Where a choice of $\pm$ and/or $u$ is listed, each choice generates a
different ring.) The values for the conductor were obtained by explicitly calculating the $\GL_2$-orbits on $V_{\Z/16\Z}$
and $V_{\Z/27\Z}$ using PARI/GP.

The densities are given as multiples
of the local densities in the table above.\footnote{The 
multipliers of $\frac{1}{27}$ occurring in the $p = 3$ table
were mistakenly printed as $\frac{1}{81}$ in \cite{R}.}
To compute these, recall that the local densities at $s = 1$ are given by
$\frac{1}{|\Disc(R)|_p |\Aut(R)|}$. Note that $|\Aut(R)| = 3$ for $R = \Z_3[x]/(x^3 - 3x^2 + 3u)$ ($u = 1, 4, 7$)
and $|\Aut(R)| = 1$ for the other extensions of $\Z_3$ listed. Moreover, it follows from our work in \cite{TT} that 
the density multipliers at $s = 5/6$ are the same as those for $s = 1$.

\begin{center}
\begin{tabular}{ l | c c  c }
Polynomial over $\Z_2$ & \textnormal{Conductor} & \textnormal{Density multiplier} \\ \hline
$x^2 + 2x \pm 2$ & $2^3$ & $\times \frac{1}{4}$ \\
$x^2 \pm 2u \ (u = 1, 3)$ & $2^4$ & $\times \frac{1}{8}$  \\
\end{tabular}
\end{center}
\begin{center}
\begin{tabular}{ l | c c  c }
Polynomial over $\Z_3$ & \textnormal{Conductor} & \textnormal{Density multiplier} \\ \hline
$x^3 \pm 3x + 3$ & $3^2$ & $\times \frac{1}{3}$ \\
$x^3 + 3x^2 + 3$ &  $3^2$ & $\times \frac{1}{9}$ \\
$x^3 - 3x^2 + 3u \ (u = 1, 4, 7)$ &  $3^3$ & $\times \frac{1}{27}$ \\
$x^3 + 3u \ (u = 1, 4, 7)$ & $3^3$ & $\times \frac{1}{27}$ \\
\end{tabular}
\end{center}

We are now prepared to prove Theorem \ref{thm_rc_extended}.
Consider a set of local specifications $\calS_p$ at a finite set of primes $\calP$.
(We also write $\calP = \prod_{p \in \calP} p.$)
For each $q$ coprime to $\calP$,
we define zeta functions
\begin{equation}
\xi^{\pm}_{\calS, q}(s) := \sum_{x \in \SL_2(\Z) \backslash V_{\Z}}
\frac{1}{|\Stab(x)|} \Phi_q(x) \Phi_{\calS}(x) |\Disc(x)|^{-s},
\end{equation}
where $\Phi_q$ and $\Phi_{\calS}$ are the characteristic functions of those $x$ nonmaximal at $q$ and satisfying
$\calS$, respectively. (Observe that maximality at $\calP$ is built into our local specifications.)

As established in \cite{TT}, and originally proved by Datskovsky and Wright, these 
zeta functions again have analytic continuations and functional equations of the same shape,
with residues
\begin{equation}\label{eqn_residue_ext}
\Res_{s = 1} \xi^{\pm}_{\calS, q}(s) = 
\alpha^{\pm} \mathscr{A}(\calS) \prod_{p | q} \bigg(
\frac{1}{p^2} + \frac{1}{p^3} - \frac{1}{p^5} \bigg)
+
\beta \mathscr{B}(\calS) \prod_{p | q} \bigg(
\frac{2}{p^2} - \frac{1}{p^4} \bigg),
\end{equation}
\begin{equation}\label{eqn_residue_ext2}
\Res_{s = 5/6} \xi^{\pm}_{\calS, q}(s) = \gamma^{\pm} \mathscr{C}(\calS) \prod_{p | q} \bigg(
\frac{1}{p^{5/3}} + \frac{1}{p^2} - \frac{1}{p^{11/3}} \bigg).
\end{equation}

The quantities $\mathscr{A}(\calS), \mathscr{B}(\calS), \mathscr{C}(\calS)$ are evaluated in \cite{TT} in terms of
certain adelic integrals.
They are naturally multiplicative (e.g., $\mathscr{A}(\calS) = \prod_{p \in \calP} \mathscr{A}(\calS_p)$),
and when $\calS_p$ is the set of all maximal cubic rings over $\Z_p$, 
we have
\begin{equation}\label{eqn_adelic_ints}
\mathscr{A}(\calS_p) = \bigg(1 - \frac{1}{p^2} \bigg) \bigg(1 - \frac{1}{p^3} \bigg), \ \
\mathscr{B}(\calS_p) = \bigg(1 - \frac{1}{p^2} \bigg)^2, \ \
\mathscr{C}(\calS_p) = \bigg(1 - \frac{1}{p^{5/3}} \bigg) \bigg(1 - \frac{1}{p^2} \bigg).
\end{equation}
For a general local specification, 
$\mathscr{A}(\calS_p)$
and $\mathscr{C}(\calS_p)$ are equal to the product of \eqref{eqn_adelic_ints} and the 
normalized local densities at $s = 1$ and at $s = 5/6$ given above.
Observe that our example in \eqref{example_density} exactly computes such an $\mathscr{A}(\calS_p)$.

To determine $\mathscr{B}(\calS_p)$ in general, we multiply the expression
in \eqref{eqn_adelic_ints} by
the {\itshape reducible} local density.
The unnormalized reducible densities are given by the following table, and we normalize them
by dividing by $1 + 1/p$.
\begin{center}
\begin{tabular}{ l | c }
Condition at $p$ & \textnormal{Reducible density} \\ \hline
\textnormal{Totally split} & $1/2$ \\ \hline
\textnormal{Partially split} & $1/2$ \\ \hline
\textnormal{Inert} & $0$ \\ \hline
\textnormal{Partially ramified} & $1/p$ \\ \hline
\textnormal{Totally ramified} & $0$ \\ 
\end{tabular}
\end{center}
The partially ramified case is divided into two or (for $p = 2$) six subcases, and the density multipliers are the same as before.

\begin{remark} The reducible density can be described in terms of the geometric interpretation given
for the irreducible ($s = 1$) density; the difference is that totally split points are counted triple and inert or
totally ramified points are not counted at all. If $\calS_p$ counts only one of these 
types of points, then $\mathscr{B}(\calS_p)$ is equal to $3 \mathscr{A}(\calS_p),$
$\mathscr{A}(\calS_p),$ or zero as appropriate. In particular, the ratio of $\mathscr{A}(\calS_p)$ and
$\mathscr{B}(\calS_p)$ in \eqref{eqn_adelic_ints} is equal to the ratio of the appropriate normalizing factors.
\end{remark}

We now ready to prove Theorem \ref{thm_rc_extended}, closely following the proof of Theorem \ref{thm_rc}.
Write $N = N(\calS) = \prod_{p \in \calP} p^{e_p}$, so that $\Phi_{\calS}(x)$ is well defined on $V_{\Z / N \Z}$.
Then the Fourier transform of $\Phi_{\calS}(x)$ is given by
\begin{equation}
\widehat{\Phi}_{\calS}(x) = \frac{1}{N^4} \sum_{y \in V_{\Z / N \Z}}
\Phi_{\calS}(y) \exp(2 \pi i [x, y]/N),
\end{equation}
and this Fourier transform appears in the dual zeta function
\begin{equation}\label{def_shintani_dual_spec}
\widehat{\xi}^{\pm}_{\calS,q}(s) := \sum_{x \in \SL_2(\Z) \backslash \widehat{V}_{\Z}} \frac{1}{|\Stab(x)|}
\widehat{\Phi}_q(N^{-1} x) \widehat{\Phi}_{\calS}(q^{-2} x) \big( |\Disc(x)| / (q^8 N^4 )\big)^{-s}.
\end{equation}
Here $N^{-1}$ and $q^{-2}$ are multiplicative inverses of $N$ and $q^2$ modulo $q$ and $N$ respectively. 
Observe that $\widehat{\Phi}_q(N^{-1} x) = \widehat{\Phi}_q(x)$.
We require a bound on $\widehat{\Phi}_{\calS}(q^{-2} x)$, but
the trivial bound 
$|\widehat{\Phi}_{\calS}(x)| \leq 1$ 
suffices for an interesting result. In particular, if we write $\widehat{\xi}^{\pm}_q(s) = \sum b_q(\mu_n) \mu_n^{-s}$
as before, then the series $\widehat{\xi}^{\pm}_{\calS, q}(s)$ is bounded coefficientwise by
$\sum |b_q(\mu_n)| (\mu_n/N^4)^{-s}$. 
In Proposition \ref{prop_high_middle_bound}
we now sum over $\mu_n/N^4 > z$ and $\mu_n/N^4 < z$ respectively, so our
bounds on the partial sums of \eqref{def_shintani_dual_spec}
are equal to $N^4$ times the bounds of Proposition \ref{prop_high_middle_bound}.

This factor of $N^4$ appears in all of our estimates involving the dual zeta function, and 
the remainder of the analysis is now essentially unchanged.
Carrying out the analysis in Section \ref{sec_proof}, we obtain, in place of \eqref{eqn_choose_Q},
\begin{equation}\label{eqn_choose_Q2}
N(X, \calS) - R_Q(X, \calS) \ll Q^{7/2 + \epsilon} N(\calS)^4 + X^{1 + \epsilon}/Q.
\end{equation}
We optimize our error term by 
choosing $Q = X^{2/9} N(\calS)^{-8/9}$. Then the right side of \eqref{eqn_choose_Q2} is $X^{7/9} N(\calS)^{8/9}$, and this is the error
term appearing in Theorem \ref{thm_rc_extended}.\footnote{There is a discrepancy between the simple definition of
$e_p$ in Theorem \ref{thm_rc_extended} and the ``correct'' definition here when $p = 2$ or $3$, but this may be
absorbed into the implied constant.} The main terms of Theorem \ref{thm_rc_extended} are obtained from the residue
formulas \eqref{eqn_residue_ext} and \eqref{eqn_residue_ext2}.

It remains to prove that the contribution from {\itshape reducible} maximal cubic rings matches
the contribution of the $\mathscr{B}(\calS)$ term in \eqref{eqn_residue_ext}.
With the single exception of $\Z^3$, the reducible cubic rings are 
the rings $\Z \times \calO_F$, where $\calO_F$ is the ring of integers of a quadratic field. This
ring may have any of the splitting types above aside from the inert or totally ramified splitting types, and
these conditions depend only on the
discriminant modulo $M := 108 \prod_p p^{e_p}$. (Here $108 = 2^2 3^3$ is the gcd of the coefficients in the formula
\eqref{eqn_disc_formula} for the discriminant; in fact, $M := 4 \prod_p p^{e_p}$ is enough.)

The number of such rings is equal to the number of squarefree integers in certain arithmetic progressions modulo $M$,
except for 
special conditions at 2.
To handle these conditions we assume that $2^6 | M$, and sum over multiple residue classes modulo $2^6$ if necessary.
This distinguishes
among the 
eight choices of $\calO_F \otimes \Z_2$, and 
the relevant quadratic fields are counted by the following lemma. The proof is a relatively straightforward
generalization of (\cite{ten}, Chapter I.3.7, Theorem 9), so we omit the detail.

\begin{lemma} Assume that $2^6 | m$. Then the number of quadratic fields $F$ with $0 < \pm \Disc(F) < X$
and $\Disc(F) \equiv a \ (\textmod \ m)$ is equal (for each choice of sign) to
\begin{equation}\label{lem_count_reduc}
\frac{8}{\pi^2 m} X e(a, 2) \prod_{\substack{p > 2 \\ p^k || m, \ k \geq 1}} e(a, p^k) \bigg(1 - \frac{1}{p^2} \bigg)^{-1} + O(\sqrt{X}),
\end{equation}
where
\begin{equation}
e(a, p^k) := \left\{
\begin{array} {l l} 1 & \hbox{if } p \nmid a, \\
1 & \hbox{if } k \geq 2 \hbox{ and } p^2 \nmid a, \\
1 - 1/p & \hbox{if } k = 1 \hbox{ and } p | a, \\
0 & \hbox{otherwise}; \end{array} \right. \ \ \ (p \neq 2)
\end{equation}
\begin{equation}
e(a, 2) = :\left\{
\begin{array} {l l}
1 & \hbox{if } a \equiv 1 \ (\textmod \ 4), \\
1 & \hbox{if } a \equiv 8, 12 \ (\textmod \ 16), \\
0 & \hbox{otherwise. } \end{array} \right.
\end{equation}

\end{lemma}

We then sum the result in \eqref{lem_count_reduc} over all appropriate residue classes. For example, if we are counting fields
split at $p$, we sum over those $a$ which are quadratic residues modulo $p$. We then check that this matches the total contribution from $\mathscr{B}(\calS)$, up to an error $\ll X^{1/2} N(\calS).$ This is smaller than our previous error term whenever our result is nontrivial, and this
completes the proof.

\begin{example} 
We illustrate our results by computing the expected number of fields $K$ with $0 < \Disc(K) < X := 2 \cdot 10^6$ which
are inert at 7 and partially ramified at 5.

Let $\calS$ denote the set of these two local specifications. Using the tables above, we compute that
\begin{equation}
C(\calS) = \frac{1/3}{1 + \frac{1}{7} + \frac{1}{49}} \cdot \frac{1/5}{1 + \frac{1}{5} + \frac{1}{25}} = \frac{245}{5301} = .046217 \cdots,
\end{equation}
\begin{equation}
K(\calS) = \frac{\frac{1}{3} \Big(1 + \frac{1}{7} \Big) \cdot \Big(1 - \frac{1}{7^{1/3}}\Big)}
{\Big(1 - \frac{1}{7^{5/3}} \Big) \Big(1 + \frac{1}{7} \Big)} \cdot 
\frac{\frac{1}{5} \Big(1 + \frac{1}{5^{1/3}} \Big)^2 \cdot \Big(1 - \frac{1}{5^{1/3}}\Big)}
{\Big(1 - \frac{1}{5^{5/3}} \Big) \Big(1 + \frac{1}{5} \Big)} = .030884 \cdots,
\end{equation}
and therefore expect to find
\begin{equation}
\approx .046217 \cdot \frac{1}{12 \zeta(3)} X  + .030884 
\frac{4 \zeta(1/3)}{5 \Gamma(2/3)^3 \zeta(5/3)} X^{5/6} = 6408.0\cdots - 812.7\cdots \approx 5595
\end{equation}
fields. Using PARI/GP to analyze the local behavior of the fields in Belabas' tables, we find that there are in fact
5546 such fields.
\end{example}

\subsection{Local conditions for 3-torsion problem in quadratic fields}\label{subsec_local_torsion}
We now turn to the proof of Theorem \ref{thm_rc_torsion2}. In this case, a local specification at $p$ consists of
a choice of $\Q(\sqrt{D}) \otimes \Q_p$. In contrast to the case of cubic fields, any local specification at a prime
$p > 2$ is determined by the residue class of 
$D \ (\textmod \ p^2)$, and a specification at 2 is determined by $D \ (\textmod \ 64)$.

As before, if $D$ is a fundamental discriminant,
subgroups of $\Cl(D)$ of index 3 are in bijection with 
cubic fields of discriminant $D$. Moreover, arguments from algebraic number theory show that
local specifications for $Q(\sqrt{D})$ correspond to local specifications for these cubic fields. In particular,
$p$ is inert in $\Q(\sqrt{D})$ if and only if it partially splits in the cubic fields, and $p$ splits if and only
if it is either inert or totally split in these cubic fields. If $p$ is ramified and $K$ is a cubic field
of discriminant $D$, then the correspondence is given by the isomorphism 
$K \otimes \Q_p \simeq (\Q(\sqrt{D}) \otimes \Q_p) \times \Q_p$.

The local densities $C'(\calS)$ and $K'^{\pm}(\calS)$ are therefore determined by the tables in Section
\ref{subsec_local_specs}.\footnote{We further check that introducing local specifications also multiplies the contribution
of the trivial element of class groups by $C'(\calS)$. In other words, the local densities of trivial and nontrivial 3-torsion
elements of the class group are the same at all finite places, but not at the infinite place. For this reason we don't write
$C'^{\pm}(\calS)$ here.}
The unnormalized densities are the same, and the normalization factors are given by
adding the unnormalized densities for all splitting types other than `totally ramified':
\begin{equation}\label{def_cpkp2}
C_p := 1 + \frac{1}{p}, \ \ \
K_p := 1 + \frac{1}{p^{1/3}} + \frac{1}{p^{2/3}} + \frac{2}{p} + \frac{2}{p^{4/3}} + \frac{1}{p^{5/3}}.
\end{equation} 

The proof is a straightforward combination of the proofs of Theorems \ref{thm_rc_torsion} and \ref{thm_rc_extended},
and the only new step occurs in our evaluation of the error term.  Again $\widehat{\xi}^{\pm}_{\calS, q}(s)$ is bounded coefficientwise by
$\sum |b_q(\mu_n)| (\mu_n/N^4)^{-s}$. This factor of $N^4$ appears in Proposition \ref{prop_high_middle_bound2};
note that the cutoffs of $q^{-3}$ and $q^{-2}$ for the ranges of $\mu_n$ there still apply to $\mu_n$ and not $\mu_n/N^4$.

As before we split into small and larger ranges of $Q$, and in the larger range 
we obtain, in place of \eqref{eqn_choose_Q}, \eqref{eqn_chooseQ_1a}, and \eqref{eqn_choose_Q2},
\begin{equation}\label{eqn_choose_Q2a}
Q^{15/8 + \epsilon} X^{3/8} N^{5/2} + Q^{7/2 + \epsilon} N^4 + X^{1 + \epsilon}/Q,
\end{equation}
where the first two terms correspond to the ranges $\mu_n < q^{-2}$ and $q^{-2} < \mu_n < z N^4$ respectively.
Our theorem follows by choosing $Q = X^{5/23} N(\calS)^{-20/23}$.

\begin{remark} To illustrate our method, we remark that we 
could formally derive Theorem \ref{thm_rc_torsion} as a consequence of Theorem
\ref{thm_rc_extended}, by imposing the local condition `not totally ramified' at every prime. This would
not treat the error terms in an acceptable manner, but this would essentially
amount to a variation of the same proof.
\end{remark}

\subsection{Arithmetic progressions}\label{subsec_ap}
This brings us to the problem of counting fields in arithmetic progressions. We wish to simultaneously allow
local specifications as in the last section, possibly to the same moduli.

As in prime number theory, we can approach this question by twisting by Dirichlet characters.
If $\chi$ is a Dirichlet character $(\textmod \ m)$, write
\begin{equation}\label{eqn_def_twisted}
N_3^{\pm}(X, \chi) := \sum_{\substack{[K : \Q] = 3 \\ 0 < \pm \Disc(K) < X}} \chi(\Disc(K)).
\end{equation}
Then
if $(a, m) = 1$, we have the usual orthogonality relation
\begin{equation}\label{eqn_def_ap}
N_3^{\pm}(X; m, a) := \sum_{\substack{[K : \Q] = 3 \\ 0 < \pm \Disc(K) < X \\ \Disc(K) \ \equiv \ a \ (\textmod \ m)}} 1
= \frac{1}{\phi(m)} \sum_{\chi \ (\textmod \ m)} \overline{\chi}(a) N_3^{\pm}(X, \chi).
\end{equation}

In addition, estimating $N_3^{\pm}(X; m, a)$ is nontrivial when $(a, m) > 1$. 
An appropriate choice of local specifications $\calS$ allows us to select exactly those cubic fields whose
discriminants are divisible by $(a, m)$, and we have
\begin{equation}
N_3^{\pm}(X; m, a) = \frac{1}{\phi(m)} \sum_{\chi \ \big(\textmod \ \frac{m}{(a, m)}\big)} \overline{\chi}\Big( \frac{a}{(a, m)} \Big)
\sum_{\substack{[K : \Q] = 3 \\ 0 < \pm \Disc(K) < X \\ K \in \calS}} \chi\bigg( \frac{\Disc(K)}{(a, m)} \bigg).
\end{equation}
Presuming we can estimate the right hand side, we obtain estimates for $N_3^{\pm}(X; m, a)$ in the same way.

We will obtain estimates for $N_3^{\pm}(X; m, a)$ for any values of $m$ and $a$, subject to an arbitrary set of local specifications.
To do so, we introduce
{\itshape orbital $L$-functions}, which are Shintani zeta functions twisted by Dirichlet characters.

{\bf Notation and conventions}. Although we will work in complete generality, we introduce several simplifying reductions which still allow
us to recover results in the general case.

As we are allowing arbitrary local specifications, it suffices to work with {\itshape primitive} characters.
Assume that we are given a primitive Dirichlet character $\chi \ (\textmod \ m)$
and a set of local specifications $\calS = (\calS_p)_{p \in \calP}$. By refining $\calS$ if necessary, we 
assume that $\calP$ includes all primes dividing $m$
and that for each $p \in \calP$, $\calS_p$ consists of a single choice for $\calO_K \otimes \Z_p$. We may
handle imprimitive characters by introducing local specifications corresponding to the condition
$\chi(n) = 0$, and we obtain results
for more general local specifications (including the set of all maximal cubic rings) by summing over the $\ll X^{\epsilon}$
possible choices of $\calS$.

We define quantities $f_p \geq 0$ by  $m = \prod_{p \in \calP} p^{f_p}$, and we define characters 
$\chi_p \ (\textmod \ p^{f_p})$ by the formula 
$\chi(u) = \prod_p \chi_p(u)$. We always regard the characters $\chi_p$ as primitive characters modulo
$p^{f_p}$, except to make sense of this formula (where we regard them as imprimitive characters modulo $m$).
In case $f_p = 0$, $\chi_p$ is the trivial character modulo 1.

We define quantities $r_p \geq 0$ to be the $p$-adic valuations of the discriminants of the cubic rings $R_p/\Z_p$ 
specified by $\calS_p$, and we write $r = \prod_p p^{r_p}$. Then our local specifications $\calS$ include a restriction
to those fields $K$ whose discriminants satisfy $r | \Disc(K)$ and $(m, \Disc(K) / r) = 1$. In particular, this implies
that  $\chi(\Disc(K)/r) \neq 0$ for each $K$ being counted.

As before, we define quantities $e_p \geq 0$ such that $\calS_p$
may be detected by reducing the Delone-Faddeev correspondence modulo $p^{e_p}$.
Finally, we define an integer $N$ (the ``conductor'') by $N = \prod_p \lcm(p^{f_p + r_p}, p^{e_p})$. Then all
of the conditions described above may be detected by reducing the Delone-Faddeev correspondence modulo $N$,
and apart from possible factors of 2 and 3, $N$ is the minimal integer with this property.

Subject to the assumptions and notation above, we define
\begin{equation}\label{eqn_def_n3}
N_3^{\pm}(X, \calS; r, \chi) := \sum_{\substack{0 < \pm \Disc(K) \leq X \\ K \in \calS}} \chi \bigg(\frac{\Disc(K)}{r} \bigg),
\end{equation}
and it is this quantity\footnote{Observe that it is redundant to include $r$ in our notation on the left, but we believe
this notation is clearest.}
 we will estimate. To do this, we introduce,
for each $q$ coprime to $N$,
an orbital $L$-function
\begin{equation}\label{eqn_def_shintani_orbital}
L^{\pm}_{\calS, q}(s, r, \chi) := \sump_{x \in \SL_2(\Z) \backslash V_{\Z} } \frac{\chi(\Disc(x) / r)}{|\Stab(x)|} |\Disc(x)|^{-s},
\end{equation}
where the sum is restricted to $x$ which are of the correct sign, nonmaximal at $q$, and satisfy the local specifications
given by $\calS$. It satisfies the same functional equation as before, with formulas for the residues and the dual zeta
function to be described later.

It was proved by Datskovsky and Wright \cite{DW2} that these $L$-functions are entire whenever $\chi^6 \neq 1$. In this case,
$N_3^{\pm}(X, \calS; r, \chi)$ will consist only of an error term. Therefore, we will assume throughout that $\chi^6 = 1$, except as
noted in the proof of Theorem \ref{thm_technical_ap}.

We may further reduce to the case where $\chi$ is a primitive {\itshape cubic} character as follows. If 
$\chi$ is not cubic, then some $\chi_p$ is not cubic, so write $\chi_p = \psi_p \phi_p$ where $\psi_p$ is cubic or trivial and
$\phi_p$ is quadratic. Recall that $\mathcal{S}$ specifies a single choice $R_p$ for the cubic ring $\calO_K \otimes \Z_p$,
and we note that $\Disc(R_p)/p^{r_p}$ is well-defined as an element of $\Z_p^{\times} / (\Z_p^{\times})^2$. It follows that
$\phi_p(\Disc(x)/r) = \phi_p(\Disc(R_p)/p^{r_p}) \phi_p(r/p^{r_p})$ for any $x$ counted in \eqref{eqn_def_shintani_orbital},
and hence
$L^{\pm}_{\calS, q}(s, r, \chi) = \phi_p(\Disc(R_p)/p^{r_p}) \phi_p(r/p^{r_p}) L^{\pm}_{\calS, q}(s, r, \chi \phi_p).$

This brings us to our most general result on cubic fields:
\begin{theorem}\label{thm_technical_ap}
Assume the notation and conventions above. Then whenever $\chi^6 \neq 1$, we have
\begin{equation}\label{eqn_technical_ap_trivial}
N_3^{\pm}(X, \calS; r, \chi) = 
O(X^{7/9 + \epsilon} N^{8/9}).
\end{equation}
When $\chi^6 = 1$, we may reduce to the case $\chi^3 =1 $ as described above, in which case we have
\begin{equation}\label{eqn_technical_ap}
N_3^{\pm}(X, \calS; r, \chi) = 
\delta(\chi) C^{\pm}(\calS)
\frac{1}{12 \zeta(3)} X + 
K^{\pm}(\calS, \chi)
\frac{4 L(1/3, \chi)}{5 \Gamma(2/3)^3 L(5/3, \chi^2)} X^{5/6} + O(X^{7/9 + \epsilon} N^{8/9}),
\end{equation}
where $\delta(\chi)$ is $1$ if $\chi$ is trivial and $0$ otherwise, $C^{\pm}(\calS)$ is as in Section \ref{subsec_local_specs}, and $K^{\pm}(\calS, \chi)$ is described below.
\end{theorem}

The quantity $K^{\pm}(\calS, \chi)$ is computed in \cite{TT}, although many of these computations are originally due to
Datskovsky and Wright (\cite{DW2}, Proposition 5.4).
We will simply state the results here. We have
$K^{\pm}(\calS, \chi) = K^{\pm} \prod_{p \in \calP} K(\calS_p, \chi)$, so it suffices to give the value of $K(\calS_p, \chi)$
for each $p$. Note that $K(\calS_p, \chi)$ depends on $\chi$ and not only $\chi_p$.

When $p$ does not divide $m$ (the conductor of $\chi$), $K(\calS_p, \chi)$ is given by dividing the appropriate value in the table
by the normalizing factor
\begin{equation}\label{def_kp_twisted}
K_{p, \chi} := \frac{(1 - \chi(p)^2 p^{-5/3})(1 + p^{-1})}{1 - \chi(p) p^{-1/3}}.
\end{equation}
\begin{center}
\begin{tabular}{ l | c | c }
Condition at $p$ & \textnormal{Density at} $s = 5/6$ & Value of $\phi_p$ ($p \neq 2$) \\ \hline
\textnormal{Totally split} & $(1 + \chi(p) p^{-1/3})^3/6$ & $1$ \\ \hline
\textnormal{Partially split} & $(1 + \chi(p) p^{-1/3})(1 + \chi(p)^2 p^{-2/3})/2$ & $-1$ \\ \hline
\textnormal{Inert} & $(1 + p^{-1})/3$  & $1$ \\ \hline
\textnormal{Partially ramified} & $(1 + \chi(p) p^{-1/3})^2 / p$ & -- \\
$\ \ (p \neq 2) \ x^2 + a u^2 p$ & $\times \frac{1}{2}$ & $\pm 1$ \\ \hline
\textnormal{Totally ramified}
& $(1+ \chi(p) p^{-1/3}) / p^2$ & \\
$\ \ (p \equiv 2 \ (\textmod \ 3))$ & $\times 1$ & $\phi_p(-3)$ \\
$\ \ (p \equiv 1 \ (\textmod \ 3)) \ x^3 + a u^3 p$ & $\times \frac{1}{3}$ & $\phi_p(-3)$
\end{tabular}
\end{center}
For convenience, we have also listed the value of $\phi_p(\Disc(R_p)/p^{r_p})$ for each row, where $\phi_p$ is
the nontrivial quadratic character mod $p$. Here 
$R_p$ is any cubic ring over $\Z_p$ of the given splitting type,
and $r_p$ is the $p$-adic valuation of $\Disc(R_p)$. 

In the partially or totally ramified cases for $p = 2$ or $3$, the density at $s = 5/6$ is given by the relative densities
given in our previous tables. The values of $\phi_3$ for totally ramified rings are given in the table below.

When $p$ does divide $m$, then either $p \equiv 1 \ (\textmod \ 3)$ or $p = 3$, and the results
are rather different. As each $\chi_p$ is primitive cubic, we note that
$\chi_p$ has conductor $p$, except $\chi_3$ which has conductor 9.

For $p \neq 3$, our results
involve (ordinary) Gauss sums
\begin{equation}
\tau_p(\chi_p) := \sum_{t \in \F_p^{\times}} \chi_p(t) e^{2 \pi i t/p}.
\end{equation}
In either case we divide the values given in the following tables by the normalizing factor
\begin{equation}
K_{p, \chi} := 1 + p^{-1}.
\end{equation}
For $p \neq 3$, we have
\begin{center}
\begin{tabular}{ l |  c | c}
Condition at $p$ & \textnormal{Density at} $s = 5/6$ & Value of $\phi_p$ $(p \neq 2)$ \\ \hline
\textnormal{Totally split} & $\tau_p(\chi_p^2) / 6p^2$ & $1$ \\ \hline
\textnormal{Partially split} & $- \tau_p(\chi_p^2) / 2p^2$ & $-1$ \\ \hline
\textnormal{Inert} &  $\tau_p(\chi_p^2) / 3p^2$ & $1$ \\ \hline
\textnormal{Partially ramified} & $\chi_p(4) \chi_{m/p}(p) p^{-4/3}$ & -- \\ 
$\ \ (p \neq 2) \ x^2 + a u^2 p$ & $\times \frac{1}{2}$ & $\pm 1$ \\ \hline
\textnormal{Totally ramified, $x^3 + a u^3 p$}
& $\big(\chi_p(a)^2 + \chi_p(a) \chi_{m/p}(p) p^{-1/3}\big)/3p^2$ & $\phi_p(-3)$
\end{tabular}
\end{center}
For $p = 3$, we have the following:
\begin{center}
\begin{tabular}{ l |  c | c}
Condition at $p$ & \textnormal{Density at} $s = 5/6$ & Value of $\phi_3$ \\ \hline
\textnormal{Totally split} & $\chi_p(4)/6p$ & $1$ \\ \hline
\textnormal{Partially split} & $\chi_p(4)/2p$ & $-1$ \\ \hline
\textnormal{Inert} &  $\chi_p(4)/3p$ & $1$ \\ \hline
\textnormal{Partially ramified} $(x^2 \pm 3)$ & 
$\pm (1 - \chi_p(2)) \chi_{m / p}(p) p^{-7/3}$ & $\mp 1$ \\ \hline
$x^3 + 3x + 3$ & $(\chi_p(4) - 1)/p^4$ & $-1$ \\
$x^3 + 6x + 3$ &  $(2\chi_p(4) + 1)/p^4$ & $1$ \\
$x^3 - 3x^2 + 3u \ (u = 1, 4, 7)$ & $\big(\chi_p(u) + \chi_{m/p}(p) p^{-1/3}\big)/p^5$ & $1$ \\
$x^3 + 3x^2 + 3$ &  $\chi_p(2) \chi_{m/p}(p)p^{-13/3}$ & $-1$ \\
$x^3 + 3u \ (u = 1, 4, 7)$ & $\chi_p(u) \chi_{m/p} (p) p^{-16/3}$ & $-1$
\end{tabular}
\end{center}
\label{tab:ap_3}
In the tables above, $\chi_{m/p}(n) := \prod_{p' | m, \ p' \neq p} \chi_{p'}(n)$.
\begin{proof}[Proof of Theorem \ref{thm_technical_ap}]
We define a test function $\Phi_N(x)$ to be
$\chi(\Disc(x)/r)$ when $x$ satisfies $\calS$, and zero otherwise.
Our assumptions ensure that $\Phi_N(x)$ is well-defined on $V_{\Z}$ and $V_{\Z / N \Z}$, 
and that $\chi(|\Disc(x)|/r) \neq 0$ for all $x$ being counted. For each $q$ coprime to $N$, we recall
the $L$-function
\begin{equation}
\xi^{\pm}_{\calS, q}(s, r, \chi) := \sum_{x \in \SL_2(\Z) \backslash V_{\Z}}
\frac{1}{|\Stab(x)|} \Phi_q(x) \Phi_N(x) |\Disc(x)|^{-s},
\end{equation}
defined in \eqref{eqn_def_shintani_orbital}.
The $L$-functions $L^{\pm}_{\calS, q}(s, r, \chi)$ again have analytic continuations
and satisfy the functional equation \eqref{eqn_shintani_FE}.

We proved the case $\chi = 1$ in Section \ref{subsec_local_specs}, so we now assume that $\chi$ is nontrivial.
Then
$L^{\pm}_{\calS, q}(s, r, \chi)$ is entire except for a pole at $s = 5/6$ if $\chi$ is cubic, in which case the residue is
\begin{equation}\label{eqn_56_res_ap}
\Res_{s = 5/6} L^{\pm}_{\calS, q}(s, r, \chi) =
\frac{2 \zeta(2) L(1/3, \chi)}{3 \Gamma(2/3)^3} K^{\pm}(\calS, \chi) \prod_{p | q}
\bigg(
\frac{\chi(p)^2}{p^{5/3}} + \frac{1}{p^2} - \frac{\chi(p)^2}{p^{11/3}} \bigg).
\end{equation}
Moreover, 
the dual $L$-function is given by
\begin{equation}\label{def_shintani_dual_AP}
\widehat{L}^{\pm}_{\calS,q}(s, r, \chi) := \sum_{x \in \SL_2(\Z) \backslash \widehat{V}_{\Z}} \frac{1}{|\Stab(x)|}
\widehat{\Phi}_q(x) \widehat{\Phi}_N(q^{-2} x) \big( |\Disc(x)| / (q^8 N^4 )\big)^{-s},
\end{equation}
where
\begin{equation}
\widehat{\Phi}_N(x) = \frac{1}{N^4} \sum_{y \in V_{\Z / N \Z}}
\Phi_N(y) \exp(2 \pi i [x, y]/N).
\end{equation}

Note in particular that \eqref{def_shintani_dual_AP} does not ``see'' the Dirichlet character apart from
$\widehat{\Phi}_N(q^{-2}x)$.

At this point we argue exactly as we did in Section \ref{subsec_local_specs}, estimating $|\widehat{\Phi}_N(q^{-2} x)| \leq 1$ as before. Everything works
in the same way, and we choose $Q = X^{2/9} N^{-8/9}$ in \eqref{eqn_choose_Q2}. When $\chi^3 = 1$ we obtain a $X^{5/6}$
term from the residue at $s = 5/6$, and when $\chi^6 \neq 1$ we obtain only the error term from \eqref{eqn_choose_Q2}.
\end{proof}
\begin{remark} Datskovsky and Wright state a version of \eqref{eqn_56_res_ap} (\cite{DW2}, p. 31), but with $L(1/3, \overline{\chi})$
in place of our $L(1/3, \chi)$. Based on \cite{TT} and our numerical computations, we believe that $L(1/3, \chi)$ is correct.
\end{remark}

\subsection{Examples and computations}\label{subsec_computations}
In this section we apply Theorem \ref{thm_technical_ap} to obtain formulas for the number of 
cubic field discriminants in arithmetic
progressions. We can do this for any arithmetic progression $a \ (\textmod \ m)$, with or without
local specifications. Unfortunately, our results
are complicated to state in general. Accordingly, we only work out the cases where 
there are no local specifications beyond those implied by our arithmetic progression, and where
either $(a, m) = 1$ or $m$ is a prime power.

We define $N_3^{\pm}(X; m, a)$ and $N_3^{\pm}(X, \chi)$ as in
\eqref{eqn_def_twisted} and \eqref{eqn_def_ap}.
Our results imply that for any $m$ and $a$,
\begin{equation}\label{eqn_simple_ap}
N_3^\pm(X; m,a)=C_1(m,a)\frac{C^\pm}{12\zeta(3)}X+K_1(m,a)\frac{4K^\pm}{5\Gamma(2/3)^3}X^{5/6}+O(m^{8/9} X^{7/9+\epsilon})
\end{equation}
for explicit constants $C_1(m, a)$ and $K_1(m, a)$. This follows from adding the result of Theorem \ref{thm_technical_ap} for
each specification at $p$, for all $p$ dividing $m$.
It remains to derive explicit formulas for $C_1(m, a)$ and $K_1(m, a)$.

\begin{remark} $C_1(m, a)$ is the density of discriminants congruent to $a$ modulo $m$, but no similar interpretation exists for $K_1(m, a)$. The
quotients $\frac{L(1/3, \chi)}{L(5/3, \chi^2)}$ are part of $K_1(m, a)$, and
$K_1(m, a)$ can be positive or negative.
\end{remark}

We begin with the case $(a, m) = 1$, which we stated in Theorem \ref{thm_progressions}.
If $4 \nmid m$, we readily deduce that
\begin{equation}
C_1(m, a) = \frac{1}{m} \prod_{p | m} \frac{1}{(1 - p^{-3})},
\end{equation}
and if $4 | m$ this is doubled or zero depending on $a \ (\textmod \ 4)$.

We turn now to $K_1(m, a)$. We only obtain contributions to \eqref{eqn_def_ap} from characters
$\chi$ with $\chi^6 = 1$, but we must consider imprimitive characters. For any $\chi$ with $\chi^6 = 1$,
we write $\chi = \psi \phi$, where $\psi$ is a primitive cubic character to a modulus dividing $m$, 
and $\phi$ is a possibly imprimitive quadratic character modulo $m$. We note however that the condition
$\chi(n) = 0$ is built into our local specifications, so the imprimitivity is irrelevant. 

We further decompose $\psi = \prod_p \psi_p$ and 
$\phi = \prod_p \phi_p$ as before,
and we have
\begin{equation}\label{eqn_ma}
K_1(m, a) = \frac{1}{\phi(m)} \sum_{\chi^6 = 1} 
\overline{\chi}(a) \frac{L(1/3, \psi)}{L(5/3, \psi^2)}
\prod_{p | m} \sum_{R_p} \phi_p(\Disc(R_p)) K(R_p, \psi),
\end{equation}
where for each $p$, the sum over $R_p$ is over all unramified cubic rings over $\Z_p$.

At this point we refer to our tables for $K(R_p, \psi)$ and make an interesting observation: For each prime $p > 3$, 
the sum over $R_p$ cancels if $\phi_p$ is trivial and $\psi_p$ is nontrivial, or vice versa. Therefore, the sum in \eqref{eqn_ma} is over
characters $\chi$ such that $\chi_p$ is of exact order 1 or 6 for each $\chi$, and for any such $\chi$ we have
$\psi = \chi^{-2}$.

We now use our tables to evaluate the sum over $R_p$. For $p \neq 3$, if $\chi_p$ is trivial, we have
\begin{equation}\label{eqn_trivial_case}
\sum_{R_p} \phi_p(\Disc(R_p)) K(R_p, \psi) = \frac{1 - \chi(p)^{-2} p^{-4/3}}{(1 - \chi(p)^2 p^{-5/3}) (1 + p^{-1})},
\end{equation}
and if $\chi_p$ is sextic then
\begin{equation}\label{eqn_sextic_case}
\sum_{R_p} \phi_p(\Disc(R_p)) K(R_p, \psi) = 
\frac{\tau_p(\chi_p^2)^3}{p^2 ( 1 + p^{-1})}.
\end{equation}
For $p = 3$, a bias appears modulo 9.
If $\psi_3$ is trivial then the sum over $R_p$ is again zero unless $\phi_3$ is also trivial, and 
\eqref{eqn_trivial_case} still holds. But if $\psi_3$ is nontrivial, the sum over $R_p$ is nonzero if
and only if $\phi_3$ is trivial. In this case, $\chi_3$ has conductor 9 and we have 
\begin{equation}\label{eqn_cubic_case}
\sum_{R_p} \phi_p(\Disc(R_p)) K(R_p, \psi) = 
\frac{\chi(4)}{4}.
\end{equation}
For $p = 2$, there are no cubic characters. There is the nontrivial quadratic character modulo 4, which we denote $\phi_4$.
 In this case we observe
that $\phi_4(R_2) = 1$ for all unramified cubic rings $R_2 / \Z_2$. Therefore $N^{\pm}_3(X, \chi)$ is the same
for $\chi = \phi_4$ and for $\chi$ the trivial character modulo 4. This reflects the fact that all field discriminants
are congruent to 1 modulo 4. 

There are two primitive quadratic characters modulo 8; fields $K$ which are totally split or totally inert at 2
have $\Disc(K) \equiv 1 \ (\textmod \ 8)$ and fields which are partially split have
$\Disc(K) \equiv 5 \ (\textmod \ 8)$. Since we restrict attention to one of these splitting types at a time,
twisting by these characters does not yield any additional information. There are no primitive
quadratic characters to moduli which are higher powers of 2, so fields equidistribute in
subprogressions modulo 16 and above.

In summary, we have proved the following:
\begin{proposition}\label{prop_ma} When $m$ is coprime to $a$ and not divisible by either 3 or 4, we have
\begin{equation}\label{eqn_ma_prop}
K_1(m, a) = \frac{1}{m} \prod_{p | m} \frac{1}{1 - p^{-2}} \sump_{\chi^6 = 1} 
\overline{\chi}(a) \frac{L(1/3, \chi^{-2})}{L(5/3, \chi^2)}
\prod_{\substack{{p | m} \\ p \nmid \textnormal{cond}(\chi)}}
 \frac{1 - \chi(p)^{-2} p^{-4/3}}{1 - \chi(p)^2 p^{-5/3}}
\prod_{\substack{{p | m} \\ p | \textnormal{cond}(\chi)}}
 \frac{\tau_p(\chi_p^2)^3}{p^2},
\end{equation}
where the sum is over primitive sextic characters $\chi$ to moduli dividing $m$, 
such that $\chi_p$ is of exact order $6$ for each $p$.

When $m$ is divisible by $3$, the same holds, except that $\chi_3$ must be of exact order $3$, 
and for $p = 3$ we substitute \eqref{eqn_cubic_case}
for \eqref{eqn_sextic_case}.

When $m$ is divisible by $4$, the above estimate is doubled if $a \equiv 1 \ (\textmod \ 4)$ and zero otherwise.
\end{proposition}
\begin{remark}
In some respects, our formula would be simpler if we summed over imprimitive characters modulo $m$.
However, we find it conceptually clearer to deal only with $L$-functions associated to 
primitive characters.
\end{remark}

\begin{remark} We used PARI/GP and Dokchitser's ComputeL \cite{dok} to compute a variety of values of
$K_1(m, a)$.
We observed that $K_1(m, a)$ behaves unpredictably with respect to factoring $m$.
As a striking example, there are more cubic field discriminants congruent to 3 than to 2 modulo 7 or 13,
but modulo $91 = 7 \cdot 13$ the pattern is reversed.

In fact the $K_1$ constants for the above progressions are all negative, so if one expects $K_1(m, a)$ to be multiplicative (it is not) then
perhaps this result is not surprising. However, we have $K_1(m, 2) < K_1(m, 3) < K_1(m, 4) < 0$ for $m = 7$ and $m = 13$,
but $K_1(91, 2)$ and $K_1(91, 4)$ are very nearly equal (and negative), and $K_1(91, 3)$ is much less than
either of these.

We further computed that $K_1(91, 5) > 0$, which shows that the secondary term can be {\itshape positive} when
restricted to arithmetic progressions.
\end{remark}

We now describe how to handle general arithmetic progressions. This is not difficult, but we do not have a particularly
elegant formulation of our results. Any exact statement would involve an enumeration of cases which is essentially
equivalent to our previous tables, so we will only give a sketch.

Consider an arithmetic progression
$ar \ (\textmod \ mr)$, where $(a, m) = 1$ but $r$ may or may not
be coprime to $m$. In this case, apart from the usual behavior at 2, $C_1(m, a)$ is equal to $\frac{1}{\phi(m)}$ times the proportion of cubic fields $K$
such that $(\Disc(K), rm) = r$. This proportion can be written as a product of local proportions over the primes dividing
$rm$, and each local proportion is determined by our previous tables. For example, if $p || r$ and $p || m$, this
local proportion is equal to
\begin{equation}
\frac{p^{-1}}{1 + p^{-1} + p^{-2}}.
\end{equation}
If we combine this with the factor of $\frac{1}{p - 1}$ coming from $\frac{1}{\phi(m)}$, we obtain a local factor of
$\frac{1}{p^2(1 - p^{-3})}$.

For the secondary term $K_1(m, a)$, \eqref{eqn_ma} still holds, except that each sum over $R_p$
is over those rings whose $p$-adic valuation is compatible with
the $p$-divisibility of $r$ and $m$. These can again be computed using our previous tables. 
 
To illustrate this, we continue our previous example. Suppose that $p || r$ and $p || m$ with
$p > 3$. Consider the
contribution of a sextic character $\chi$ to \eqref{eqn_ma}, and write $\chi = \prod_p \chi_p = \prod_p \psi_p \phi_p$
as before. If the quadratic part $\phi_p$ is nontrivial, then
$\sum_{R_p} \phi_p(\Disc(R_p)) K(R_p, \psi) = 0$ whether $\psi_p$ is trivial or not.

If $\psi_p$ and $\phi_p$ are both trivial, then 
\begin{equation}
\sum_{R_p} \phi_p(\Disc(R_p)) K(R_p, \psi) = 
\frac{(1 - \psi(p)^2 p^{-2/3})(1 + \psi(p) p^{-1/3})}{(1 - \psi(p)^2 p^{-5/3})(p + 1)}.
\end{equation} 
If $\phi_p$ is trivial and $\psi_p$ is nontrivial, then
\begin{equation}\label{eqn_nontriv}
\sum_{R_p} \phi_p(\Disc(R_p)) K(R_p, \psi) = 
\frac{\psi_p(4) \psi_{m/p}(p) p^{-4/3}}{1 + p^{-1}}.
\end{equation}

We summarize our results in the following proposition. The reader should beware that this result is
misleadingly simple,
as it obscures the distinction between $\psi$ and $\psi_p$, 
but we emphasize that we can
obtain results for other progressions in an exactly similar fashion.
\begin{proposition} For $p > 3$ and $(a, p) = 1$ we have
\begin{equation}
C_1(p^2, ap) = \frac{1}{p^2 (1 - p^{-3})},
\end{equation}
\begin{equation}\label{eqn_p2_ap}
K_1(p^2, ap) = \frac{1}{p^2 - 1}
\bigg( \frac{\zeta(1/3)}{\zeta(5/3)} \frac{(1 - p^{-2/3})(1 + p^{-1/3})}
{1 - p^{-5/3}} + \frac{1}{p^{1/3}} \sump_{\substack{\psi^3 = 1 \\ \psi \neq 1}}
\overline{\psi}(2a) \frac{L(1/3, \psi)}{L(5/3, \psi^2)} \bigg).
\end{equation}
\end{proposition}
We also obtain in the same manner (again for $p > 3$ and $(a, p) = 1$)
\begin{equation}
C_1(p^3, ap^2) = \frac{1 + \phi_p(-3a)}{p^3 (1 - p^{-3})}, \ \ \ 
K_1(p^3, ap^2) = \frac{(1 + \phi_p(-3a))(1 - p^{-2/3})}
{p^3 (1 - p^{-2}) (1 - p^{-5/3})}
\frac{\zeta(1/3)}{\zeta(5/3)},
\end{equation}
and if we further increase the powers of $p$ in the moduli of any of the previous four equations, then we introduce
no new sextic characters and hence we simply divide each term by the appropriate power of $p$.
Moreover, for $p > 3$ there are no cubic fields with discriminants divisible by $p^3$, and hence we have
completely determined the distribution of cubic field discriminants modulo powers of $p$.

For $p = 2$, there are no cubic characters, and cubic fields equidistribute in subprogressions of the arithmetic
progression corresponding to each local specification at 2.
For $p = 3$ the analysis is rather lengthy, and discriminants of cubic fields can
have 3-adic valuation as large as 5. In the interest of space we will not work out the details here;
the idea is that arithmetic progressions $a 3^k \ (\textmod \ 3^{k + 1})$ correlate with local specifications
at 3, progressions $a 3^k \ (\textmod \ 3^{k + 2})$ exhibit a bias due to the primitive cubic characters modulo 9,
and progressions $a 3^k \ (\textmod \ 3^{k + j})$ do not exhibit additional bias for $j \geq 3$.

We now illustrate our results with numerical data on
the distribution of field discriminants in arithmetic progressions modulo 7 and powers of 7.
The ``expected'' counts are the two main terms of \eqref{eqn_simple_ap},
and the actual counts were determined from Belabas' tables \cite{bel_fast}.

\begin{align*}
C_1(7,a)&=0.00993261\dots,
\\
K_1(7,a)&=
\begin{cases}
-0.0101147\dots&a=5,\\
-0.0149070\dots&a=1,\\
-0.0159463\dots&a=4,\\
\end{cases}
&K_1(7,a)&=
\begin{cases}
-0.0255309\dots&a=3,\\
-0.0265702\dots&a=6,\\
-0.0313625\dots&a=2.\\
\end{cases}
\end{align*}
\[
\begin{array}{c|c|c}
a&N_3^+(2\cdot10^6,7,a)& \textnormal{Expected} \\ 
\hline
1 & 17229 &17209\\
2 & 14327 &14277\\
3 & 15323 &15316\\
4 & 17027 &17024\\
5 & 18058 &18063\\
6 & 15150 &15131\\
\hline
\end{array}
\quad
\begin{array}{c|c|c}
a&N_3^-(10^6,7,a)&  \textnormal{Expected} \\
\hline
1 & 27281 &27216\\
2 & 24343 &24366\\
3 & 25389 &25376\\
4 & 27035 &27036\\
5 & 28051 &28046\\
6 & 25227 &25196\\
\hline
\end{array}
\]
\begin{align*}
C_1(49,7a)&=0.00141894\dots, &
K_1(49,7a)&=
\begin{cases}
-0.00159849\dots & a=3,4,\\
-0.00382342\dots & a=1,6,\\
-0.00520755\dots & a=2,5.\\
\end{cases}
\end{align*}
\[
\begin{array}{c|c|c}
a&N_3^+(2\cdot10^6,49,a)& \textnormal{Expected} \\
\hline
7 & 2155 &2157\\
14 & 1920 &1910\\
21 & 2562 &2553\\
28 & 2519 &2553\\
35 & 1921 &1910\\
42 & 2159 &2157\\
\hline
\end{array}
\quad
\begin{array}{c|c|c}
a&N_3^-(10^6,49,a)&  \textnormal{Expected} \\
\hline
7 & 3555 &3595\\
14 & 3362 &3355\\
21 & 3967 &3980\\
28 & 3980 &3980\\
35 & 3345 &3355\\
42 & 3590 &3595\\
\hline
\end{array}
\]

\begin{align*}
C_1(343,49a)&=
\begin{cases}
0.000405412\dots & a=1,2,4,\\
0&a=3,5,6,\\\end{cases}
&K_1(343,49a)&=
\begin{cases}
-0.000664801\dots&a=1,2,4,\\
0&a=3,5,6.\\\end{cases}
\end{align*}
\[
\begin{array}{c|c|c}
a&N_3^+(2\cdot10^6,343,a)&  \textnormal{Expected} \\
\hline
49 	& 697 	&692\\
98 	& 690 	&692\\
147 	& 0 	&0\\
196 	& 707 	&692\\
245 	& 0 	&0\\
294 	& 0 	&0\\
\hline
\end{array}
\quad
\begin{array}{c|c|c}
a&N_3^-(10^6,343,a)& \textnormal{Expected} \\
\hline
49 	& 1117 	&1101\\
98 	& 1092 	&1101\\
147 	& 0 	&0\\
196 	& 1083 	&1101\\
245 	& 0 	&0\\
294 	& 0 	&0\\
\hline
\end{array}
\]

\subsection{3-torsion in quadratic fields}\label{subsec_torsion_ap}
We come now to the analogue of Theorem \ref{thm_technical_ap} for 3-torsion in quadratic fields, and a discussion of 3-torsion
in arithmetic progressions. The results are quite similar, so we will keep our discussion brief. Write
\begin{equation}
M_3^{\pm}(X, \calS; r, \chi) := \sump_{\substack{0 < \pm \Disc(K) \leq X \\ K \in \calS}} \chi \bigg(\frac{\Disc(K)}{r} \bigg),
\end{equation}
as in \eqref{eqn_def_n3}, but with the restriction to fields $K$ which are nowhere totally ramified. (Recall that these
are in bijection with pairs of nontrivial 3-torsion elements in $\Cl(\Q(\sqrt{\Disc(K)}))$.
We adopt all of the notation of Section \ref{subsec_ap}, and make all of the same assumptions on $\calS$, $r$, and $\chi$.
We will prove the following theorem:

\begin{theorem}\label{thm_technical_torsion_ap}
Whenever $\chi^6 \neq 1$, we have
\begin{equation}\label{eqn_technical_ap_trivial_torsion}
N_3^{\pm}(X, \calS; r, \chi) = 
O(X^{18/23 + \epsilon} N^{20/23} ).
\end{equation}
When $\chi^6 = 1$, we may reduce to the case $\chi^3 =1 $ as described in Section \ref{subsec_ap}, in which case we have
\begin{multline}\label{eqn_technical_ap_torsion}
M_3^{\pm}(X, \calS; r, \chi) = 
\frac{\delta(\chi) C'^{\pm}(\calS)}{2 \pi^2} X + \\
K'^{\pm}(\calS, \chi)
\frac{4 L(1/3, \chi)}{5 \Gamma(2/3)^3}
\prod_{p \nmid \textnormal{cond}(\chi)} \Bigg( 1 - \frac{ \chi(p)^{-1} p^{1/3} + 1}{p (p + 1)} \Bigg) X^{5/6} + O(X^{18/23 + \epsilon} N^{20/23} ),
\end{multline}
where $\delta(\chi)$ is $1$ if $\chi$ is trivial and $0$ otherwise, and $C'^{\pm}(\calS)$ and $K'^{\pm}(\calS, \chi)$ are described below.
\end{theorem}

The proof is a straightforward combination of the proofs of Theorems  \ref{thm_rc_torsion} and \ref{thm_technical_ap}. The 
constants $C'^{\pm}(\calS)$ and $K'^{\pm}(\calS, \chi)$ are computed in the same way. We normalize $C'^{\pm}(\calS)$
by dividing by $1 + p^{-1}$, and we normalize $K'^{\pm}(\calS, \chi)$ by dividing by
\begin{equation}\label{eqn_kp_new}
K_{p, \chi} := 1 + \frac{\chi(p)}{p^{1/3}} +
\frac{\chi(p)^2}{p^{2/3}} + \frac{2}{p} + \frac{2 \chi(p)}{p^{4/3}} + \frac{\chi(p)^2}{p^{5/3}}
\end{equation}
for each $p$ dividing $N$ for which $\chi_p$ is trivial (compare with \eqref{def_cpkp2}), and
\begin{equation}
K_{p, \chi} := 1 + p^{-1}
\end{equation}
for each prime $p$ for which $\chi_p$ is nontrivial.

One can now compute as many examples as before, and one finds similar biases in arithmetic progressions to the same moduli.
For brevity's sake we will confine ourselves to a discussion of $M_3^{\pm}(X; m, a)$ (defined in the obvious manner) when
$(m, 6a) =1 $.
We have, similarly to \eqref{eqn_simple_ap},
\begin{equation}\label{eqn_simple_ap2}
M_3^\pm(X; m,a)=
C'_1(m,a)
\frac{C'^{\pm}}{2 \pi^2} X 
+ K'_1(m, a)
\frac{4 K^{\pm}}{5 \Gamma(2/3)^3}
X^{5/6}
\end{equation}
for explicit constants $C'_1(m, a)$ and $K'_1(m, a)$. If $(a, 6m) = 1$, then
\begin{equation}
C'_1(m, a) = \frac{1}{m} \prod_{p | m} \frac{1}{(1 - p^{-2})}.
\end{equation}
To evaluate $K'_1(m, a)$, we again decompose any nontrivial $\chi$ into a primitive cubic character $\psi$ and a quadratic character $\phi$,
and we have
\begin{equation}\label{eqn_ma2}
K'_1(m, a) = \frac{1}{\phi(m)} \sum_{\chi^6 = 1} 
\overline{\chi}(a) L(1/3, \psi)
\prod_{p \nmid \textnormal{cond}(\chi)} \Bigg( 1 - \frac{ \psi(p)^{-1} p^{1/3} + 1}{p (p + 1)} \Bigg)
\prod_{p | m} \sum_{R_p} \phi_p(\Disc(R_p)) K(R_p, \psi),
\end{equation}
where the rightmost sum is over all unramified cubic rings over $\Z_p$. This sum is the same as in the problem of counting cubic
fields, except for the new normalization factors.
When $(m, 6a) = 1$, this implies (as before) that the outer sum is 
over characters $\chi$ for which each $\chi_p$ has exact order 1 or 6, and that
$\psi = \chi^{-2}$ for each such $\chi$.

We readily deduce Theorem \ref{thm_torsion_progressions}, and we could deduce other variations as well.


\begin{thebibliography}{99}

\bibitem{bel_fast} K. Belabas,
\emph{A fast algorithm to compute cubic fields},
Math. Comp. \textbf{66} (1997), no. 219, 1213--1237; accompanying software available at
\url{http://www.math.u-bordeaux1.fr/~belabas/research/software/cubic-1.0.tgz}.

\bibitem{bel} K. Belabas,
\emph{On quadratic fields with large $3$-rank},
Math. Comp. \textbf{73} (2004), no. 248, 2061--2074.

\bibitem{BBP} K. Belabas, M. Bhargava, and C. Pomerance,
\emph{Error estimates in the Davenport-Heilbronn theorems}, Duke Math. J. \textbf{153} (2010), 
no. 1, 173--210.

\bibitem{BF} K. Belabas and E. Fouvry, 
\emph{Sur le 3-rang des corps quadratiques de discriminant premier ou pseudo-premier}, 
Duke Math. J. \textbf{98} (1999), pp. 217-268.

\bibitem{B} M. Bhargava,
\emph{Simple proofs of the Davenport-Heilbronn theorems}, preprint.

\bibitem{B_quartic} M. Bhargava, 
\emph{The density of discriminants of quartic rings and fields},
Ann. of Math. (2) \textbf{162} (2005), no. 2, 1031--1063. 

\bibitem{BST} M. Bhargava, A. Shankar, and J. Tsimerman,
\emph{On the Davenport-Heilbronn theorem and second order terms}, preprint.

\bibitem{CN} K. Chandrasekharan and R. Narasimhan,
\emph{Functional equations with multiple gamma factors and the average order of arithmetical
functions}, Ann. Math (2) \textbf{76} (1962), 93--136.

\bibitem{C} H. Cohn,
\emph{The density of abelian cubic fields}, 
Proc. Amer. Math. Soc. \textbf{5} (1954), 476-477.

\bibitem{DW2} B. Datskovsky and D. Wright,
\emph{The adelic zeta function associated to the space of binary cubic forms. II. Local theory},
J. Reine Angew. Math. \textbf{367}  (1986), 27--75.

\bibitem{DW3} B. Datskovsky and D. Wright,
\emph{Density of discriminants of cubic extensions},
J. Reine Angew. Math. \textbf{386}  (1988), 116--138. 

\bibitem{D} H. Davenport,
\emph{Multiplicative number theory},
Springer-Verlag, New York, 2000.

\bibitem{DH} H. Davenport and H. Heilbronn,
\emph{On the density of discriminants of cubic fields. II}, 
Proc. Roy. Soc. London Ser. A \textbf{322}  (1971), no. 1551, 405--420. 

\bibitem{DF} B. N. Delone and D. K. Faddeev,
\emph{The theory of irrationalities of the third degree (English translation},
AMS, Providence, 1964.

\bibitem{dok} T. Dokchitser, 
\emph{ComputeL}, 
PARI/GP package available at 
\url{http://www.dpmms.cam.ac.uk/~td278/computel/index.html}. 

\bibitem{EV} J. Ellenberg and A. Venkatesh,
\emph{Reflection principles and bounds for class group torsion},
Int. Math. Res. Not. no. 1 (2007), Art. ID rnm002.

\bibitem{GGS} W. T. Gan, B. Gross, and G. Savin,
\emph{Fourier coefficients of modular forms on $G_2$},
Duke Math. J. \textbf{115} (2002), 105--169.

\bibitem{H} B. Hough,
\emph{Equidistribution of Heegner points associated to the 3-part of the class group},
preprint.

\bibitem{JRW} M. J. Jacobson Jr., S. Ramachandran, and H. C. Williams,
\emph{Numerical results on class groups of imaginary quadratic fields}, ANTS VII, Berlin, 87--101,
Lecture Notes in Comput. Sci. \textbf{4076}, Springer, Berlin, 2006.

\bibitem{james} K. James,
\emph{$L$-series with nonzero central critical value},
J. Amer. Math. Soc. \textbf{11} (1998), no. 3, 635--641.

\bibitem{JR} J. Jones and D. Roberts,
\emph{A database of local fields},
J. Symbolic Comput. \textbf{41} (2006), no. 1, 80--97; accompanying database available online at
\url{http://math.asu.edu/~jj/localfields/}.

\bibitem{L} E. Landau,
\emph{\"Uber die Anzahl der gitterpunkte in gewissen Bereichen},
G\"ott. Nachr. (1912), 687--771.

\bibitem{M} A. Morra,
\emph{Comptage asymptotique et algorithmique d'extensions cubiques relatives} (text in English),
thesis, Universit\'e Bordeaux I, 2009. Available online at
\url{http://perso.univ-rennes1.fr/anna.morra/these.pdf}.

\bibitem{N} J. Nakagawa, 
\emph{On the relations among the class numbers of binary cubic forms},
Invent. Math. \textbf{134} (1998), no. 1, 101--138. 

\bibitem{NH} J. Nakagawa and K. Horie,
\emph{Elliptic curves with no rational points},
Proc. Amer. Math. Soc. \textbf{104} (1988), no. 1, 20--24.

\bibitem{O} Y. Ohno, 
\emph{A conjecture on coincidence among the zeta functions associated with the space of binary cubic forms},
Amer. J. Math. \textbf{119} (1997), no. 5, 1083--1094. 
 
\bibitem{OT} Y. Ohno and T. Taniguchi,
\emph{Relations among Dirichlet series whose coefficients are class numbers of binary cubic forms II},
preprint.

\bibitem{OTW} Y. Ohno, T. Taniguchi, and S. Wakatsuki,
\emph{Relations among Dirichlet series whose coefficients are class numbers of binary cubic forms}, 
Amer. J. Math. \textbf{131}-6 (2009), 1525--1541.

\bibitem{pari} 
PARI/GP, version {\tt 2.3.4}, Bordeaux, 2008, available from \url{http://pari.math.u-bordeaux.fr/}.

\bibitem{R} D. Roberts,
\emph{Density of cubic field discriminants},
Math. Comp. \textbf{70}  (2001),  no. 236, 1699--1705.

\bibitem{sato} F. Sato,
\emph{On functional equations of zeta distributions},
Adv. Studies in Pure Math. \textbf{15} (1989), 465--508.

\bibitem{SK} M. Sato and T. Kimura,
\emph{A classification of irreducible prehomogeneous vector spaces and their relative invariants}.
Nagoya Math. J. \textbf{65} (1977), 1--155. 

\bibitem{SS} M. Sato and T. Shintani,
\emph{On zeta functions associated with prehomogeneous vector spaces},
Ann. of Math. (2) \textbf{100}  (1974), 131--170. 

\bibitem{S} T. Shintani,
\emph{On Dirichlet series whose coefficients are class numbers of integral binary cubic forms},
J. Math. Soc. Japan  \textbf{24} (1972), 132--188. 

\bibitem{shin_reducible} T. Shintani,
\emph{On zeta-functions associated with the vector space of quadratic forms},
J. Fac. Sci. Univ. Tokyo Sect. I A Math. \textbf{22} (1975), 25--65.

\bibitem{sound} K. Soundararajan,
\emph{Divisibility of class numbers of imaginary quadratic fields},
J. London Math. Soc. (2) \textbf{61} (2000), 681--690.

\bibitem{sage} W. A. Stein et al.,
\emph{Sage Mathematics Software (Version 4.5.3),}
   The Sage Development Team, 2010, \url{http://www.sagemath.org}.
   
\bibitem{tani_simple} T. Taniguchi,
\emph{On the zeta functions of prehomogeneous vector spaces for a pair of simple algebras}, 
Ann. Inst. Fourier. \textbf{57}, no. 4 (2007), 1331--1358.

\bibitem{TT}T. Taniguchi and F. Thorne,
\emph{Orbital $L$-functions for the space of binary cubic forms}, in preparation.

\bibitem{ten} G. Tenenbaum,
\emph{Introduction to analytic and probabilistic number theory},
Cambridge University Press, Cambridge, 1995.



\bibitem{vatsal} V. Vatsal,
\emph{Canonical periods and congruence formulae},
Duke Math. J. \textbf{98} (1999), 397--419.

\bibitem{wong} S. Wong,
\emph{Elliptic curves and class number divisibility},
Int. Math. Res. Not. \textbf{12} (1999), 661-672.

\bibitem{DW1} D. Wright,
\emph{The adelic zeta function associated to the space of binary cubic forms. I. Global theory},
Math. Ann.  \textbf{270}  (1985),  no. 4, 503--534.

\bibitem{WY} D. Wright and A. Yukie,
\emph{Prehomogeneous vector spaces and field extensions},
Invent. Math. \textbf{110} (1992), no. 2, 283--314.

\bibitem{Y} A. Yukie,
\emph{Shintani zeta functions},
London Mathematical Society Lecture Note Series, 183, Cambridge University Press, Cambridge, 1993.
\end{thebibliography}
\end{document}